\documentclass[11pt,a4paper]{article}

\usepackage[latin1]{inputenc}

\usepackage{amsmath,amscd,amssymb,latexsym,graphicx,color,mathrsfs}

 \usepackage[all]{xy}
\newtheorem{theorem}{Theorem}[section]

\newtheorem{proposition}[theorem]{Proposition}
\newtheorem{corollary}[theorem]{Corollary}

\newtheorem{lemma}[theorem]{Lemma}
\newtheorem{definition}[theorem]{Definition}
\newtheorem{remark}[theorem]{Remark}

\newtheorem{example}[theorem]{Example}

\numberwithin{equation}{section}

\newcommand{\cqfd}{\hfill{\small $\Box$}}
\newenvironment{proof}[1][]{{\bf Proof #1 : }}{\begin{flushright}
\cqfd\end{flushright}}

\reversemarginpar

\newcommand{\gr}{\mathscr{G}}
\newcommand{\go}{\mathscr{G} ^{(0)}}

\newcommand\tgt[1]{{}^{T}\kern-1pt #1}
\newcommand\adi[1]{{}^{ad}\kern-1pt #1}

\title{Topological $K$-theory for discrete groups and Index theory}

\author{P. Carrillo Rouse\footnote{Institut de Math\'ematiques de Toulouse, Universit\'e Toulouse III, 118, Route de Narbonne, 31400
Toulouse, France. Email: paulo.carrillo@math.univ-toulouse.fr}, B.L. Wang\footnote{Mathematical Sciences Institute,  Australian National University, Canberra,  ACT 2601, Australia. Email: bai-ling.wang@anu.edu.au} and H. Wang\footnote{School of Mathematical Sciences, East China Normal University, Shanghai 200241, China. \newline Email: wanghang@math.ecnu.edu.cn}}

\date{\today}

\begin{document}

\maketitle

\begin{abstract}
For any countable discrete group $\Gamma$ (without any further assumptions on it) we first construct an explicit morphism from  the Left-Hand side of the Baum-Connes assembly map, $K^*_{top}(\Gamma)$, to the periodic cyclic homology of the group algebra, $HP_*(\mathbb{C}\Gamma)$. This morphism, called the Chern-Baum-Connes assembly map, allows in particular to give a proper and explicit formulation for a Chern-Connes pairing
\begin{equation}\nonumber
K^*_{top}(\Gamma)\times HP^*(\mathbb{C}\Gamma)\longrightarrow \mathbb{C}.
\end{equation}

Several theorems are needed to formulate the Chern-Baum-Connes assembly map. In particular we establish a delocalised Riemann-Roch theorem, the wrong way functoriality for periodic delocalised cohomology for $\Gamma$-proper actions, the construction of a Chern morphism between the Left-Hand side of Baum-Connes and a delocalised cohomology group associated to $\Gamma$ which is an isomorphism once tensoring with $\mathbb{C}$, and the construction of an explicit cohomological assembly map between the delocalised cohomology group associated to $\Gamma$ and the homology group $H_*(\Gamma,F\Gamma)$ (where $F\Gamma$ is the complex vector space freely generated by the set of elliptic elements in $\Gamma$). This last group $H_*(\Gamma,F\Gamma)$ identifies, by the work of Burghelea, as a direct factor of the cyclic periodic homology of the group algebra.

Moreover, in this paper $K^*_{top}(\Gamma)$ stands for the topological $K$-theory for discrete groups as originally proposed by Baum and Connes in their first paper, where wrong way functoriality in equivariant $K$-theory was used for its construction. As part of our results we prove that this model is indeed isomorphic to the analytic model for the left-hand side of the assembly map.

We conclude the paper by giving an explicit index theoretical formula for the above mentioned pairing (without any further assumptions on $\Gamma$) in terms of pairings of invariant forms, associated to geometric cycles and given in terms of delocalized Chern and Todd classes, and currents naturally associated to group cocycles using Burghelea's computation. This gives a complete solution, for discrete countable groups, to the problem of defining and computing a geometric pairing between the left hand side of the Baum-Connes assembly map, given in terms of geometric cycles associated to proper actions of the discrete group on manifolds, and cyclic periodic cohomology of the group algebra.

\end{abstract}

\pagebreak

\tableofcontents

\section{Introduction}

Let $\Gamma$ be a countable discrete group. In \cite{BCfete,BCKtop, BC}, Baum and Connes initiated a program on index theory for $\Gamma$-manifolds that lead to the so-called Baum-Connes assembly map associated to the group $\Gamma$ and to the ensuing  Baum-Connes conjecture that states that the assembly map is an isomorphism. The conjecture is at the present time still open, even though it has been proved (or in some cases just the injectivity or the surjectivity) for a large class of groups. One of the main reasons this conjecture became an important source of research is that it implies some other important conjectures in geometry, topology and analysis, for example, to mention the most famous ones, the Novikov conjecture, the stable Gromov-Lawson conjecture, the modified trace conjecture and the Kadison-Kaplansky conjecture; See for instance \cite{KL,Khal,Concg}. The assembly map in question is a group morphism
\begin{equation}
\mu: K^*_{top}(\Gamma)\longrightarrow K_*(C_r^*(\Gamma))
\end{equation}
where the left hand side $K^*_{top}(\Gamma)$ is a group constructed by assembling (as we will explain below with more details) the topological $K$-theory groups associated to $\Gamma$-proper manifolds ($spin^c$ and cocompact) while the right hand side stands for the topological $K$-theory of the reduced group $C^*$-algebra of the group. This assembly morphism has been also studied extensively in recent years in the context of higher secondary invariants, Gromov-Lawson indices, problems related to (existence and classification of) positive scalar curvature metrics and surgery on manifolds, for example see \cite{Ben}, \cite{BenRoy1}, \cite{CWXY}, \cite{DeeGof2}, \cite{HigRoe2010}, \cite{PSZ}, \cite{WXY} or \cite{Zenadv}, just to mention some recent developments related in some way to our paper.

Now, in the original papers of Baum and Connes, ref.cit., the cycles for the left hand side consist of classes of couples $(M,x)$, where $M$ is a $\Gamma$-proper, $spin^c$, cocompact manifold and $x$ is an element of the equivariant $K$-theory group\footnote{A model for this can be the $K$-theory group of the $C^*$-algebra $C^*_r(M\rtimes \Gamma)$ associated to the action groupoid $M\rtimes \Gamma$.} $K^*_{\Gamma}(M)$ and the equivalence relation is generated by being equal up to a $\Gamma$-family index, to be more precise, if $f:M\to N $ is an appropriate $\Gamma$-equivariant map between manifolds as above, then we shall have $(M,x)\sim (N,f_!(x))$ where the shriek map $f_!$ is a $\Gamma$-family $K$-theoretical index morphism.  In the original papers of Baum and Connes the assembly was then defined to be
$$[M,x]\mapsto \pi_M!(x)\in K_*(C_r^*(\Gamma))$$
where $\pi_M!(x)$ is the $\Gamma$-higher index associated to the projection $\pi_M: M\to \{*\}$.
Now, even if Baum and Connes gave the fundamental ideas and directions to follow, not all proofs were completed, but just sketched. For example, in those papers the well-definedness of the left hand side groups was sketched and not proved. The fact that it is well-defined is a consequence of the wrong way functoriality theorem proved in a more general context in \cite{CaWangAENS}. This wrong way functoriality also implies that the originally defined Baum-Connes assembly map is well-defined.

After the original formulation mentioned in the work of Baum and Connes, several different directions occurred; In fact, the assembly map was formalised some years later by Baum, Connes and Higson in \cite{BCH} using the powerful tool of Kasparov's $KK$-theory to properly define a ``Left hand side" and the assembly map, with the time other models for this left hand side also entered into the game (Baum-Douglas model, L\"uck model). The emergence of $KK$-theory led to deep and powerful results on the injectivity and surjectivity of the assembly map and allowed to establish several properties. One of the main motivations of the original ideas of Baum-Connes was to be able to give geometric formulas for the higher index, to be able to use classical topological tools for computing higher analytical indices, however actual index formulae computations of the assembly map are rare, for example in \cite{BCfete} Baum and Connes sketched a construction of a Chern character from the left hand side that would eventually lead to an actual pairing with higher currents or higher traces. In this paper we overcome this problem and fill this gap by constructing a Chern character type morphism
\begin{equation}\label{eqChernassemblyINTRO}
\xymatrix{
K_{top}^*(\Gamma)\ar[r]^-{ch_\mu^{top}}&HP_*(\mathbb{C}\Gamma)
}
\end{equation}
that will be called the Chern-Baum-Connes assembly map because it is constructed by assembling as Baum-Connes, and as will be detailed below, different Chern characters. We hence complete the program on the Chern character for discrete groups as stated in \cite{BCfete} and \cite{BCKtop} by Baum and Connes. Besides, this geometric approach allows us to continue this program by first comparing the models we use in this paper with more common analytic models used  in the literature and second by giving a complete solution, for countable discrete groups, to the problem of defining and computing a geometric pairing between the left hand side of the Baum-Connes assembly map, given in terms of geometric cycles associated to proper actions of the discrete group, and the cyclic periodic cohomology of the group algebra..
We now explain in more details the contents of our work.

For a  $\Gamma$-proper manifold $M$,  there are periodic delocalised cohomology groups $H^*_{\Gamma,deloc}(M)$, even and odd, that might be defined by periodizing the de Rham cohomology of the so-called inertia groupoid (full details in Sections \ref{sectionLiegrpdcohomology} and \ref{delocpushforwardsubsection}). In fact, one justification of why this is the correct group to consider is because there is a $\Gamma$-Chern character morphism

\begin{equation}\label{CherndelocalisedINTRO}
ch_M^\Gamma :K^*_{\Gamma}(M)\longrightarrow  H^*_{\Gamma,deloc}(M)
\end{equation}
given explicitly as the composition of the Connes-Chern character\footnote{Defined a priori in $K_*(C_c^\infty(M\rtimes \Gamma))$ but $C_c^\infty(M\rtimes \Gamma)\subset C^*(M\rtimes \Gamma)$ is stable under holomorphic calculus since the action is proper.}
\begin{equation}
Ch :K^*_{\Gamma}(M)\longrightarrow  HP_*(C_c^\infty(M\rtimes \Gamma)).
\end{equation}
followed by an isomorphism
\begin{equation}
\xymatrix{
HP_*(C_c^\infty(M\rtimes \Gamma))\ar[r]^-{TX}_-{\cong}& H^*_{\Gamma,deloc}(M).}
\end{equation}
explicitly described in \cite{TuXuChern} by Tu and Xu, and recalled in Appendix \ref{appendixTuXu}. In fact, Tu and Xu proved even more in \cite{TuXuChern} that the Chern character (\ref{CherndelocalisedINTRO}) is an isomorphism up to tensoring with $\mathbb{C}$.

From our perspective there are two major technical problems in index theory for $\Gamma$-proper actions on smooth manifolds:
\begin{enumerate}
\item Relate different $\Gamma$-proper actions and their pairings as above. Explicitly this means being able to prove (and give sense to) Riemann-Roch type theorems and pushforward functoriality for such actions and then being able to assemble these results to give an explicit formulation of a pairing
\begin{equation}
K^*_{top}(\Gamma)\times HP^*(\mathbb{C}\Gamma)\to \mathbb{C}
\end{equation}
that only depends on the group $\Gamma$ and that contains/exhausts all possible index pairings for $\Gamma$-proper smooth cocompact manifolds.
\item Compute explicit index formulas for the pairing above. In particular, as we will discuss below, the cyclic periodic cohomology groups $HP^*(\mathbb{C}\Gamma)$ can be computed explicitly, by the work of Burghelea \cite{Bur}, in terms of group cohomology groups associated to the centralizers of elliptic and non elliptic elements. The problem of computing an explicit formula for the pairing above can be stated as giving a geometric/topological expression for the pairings
\begin{equation}
\langle [M,x],\tau\rangle \in \mathbb{C}
\end{equation}
for a geometric cycle $[M,x]\in K^*_{top}(\Gamma)$ and a group cocycle $\tau$ in a group cohomology  (that only depends on its conjugacy class) of some elliptic element (as we will see it is zero for non-elliptic elements). There are several particular cases computed in the literature, for free and proper actions and localised at the unit cocycle  by Connes and Moscovici \cite{CMnov}, for particular Connes-Moscovici cocyles  by Ponge and Wang in \cite{PW1,PW2}, for some delocalised traces but under some extra hypothesis on the extensibility of the cocycles by the two last named authors in \cite{WangWangJDG}, for mention some of them, see also \cite{PPTtransverse,PR}. We are certainly aware of several other computations but as far as we know there is not a general and systematic procedure or formula that computes the above pairing.
\end{enumerate}

In this article we will completely solve the two problems above. In order to be more precise, we need to introduce some terminology and explain what are the results leading to the solution of these problems.

We now explain the main contents of this paper. Let $f:M\longrightarrow N$ be a $\Gamma$-equivariant $C^\infty$-map between two $\Gamma$-proper cocompact manifolds $M$ and $N$.
We define (Definition~\ref{defpushforwardCPH}) a pushforward (shriek) map in periodic delocalised cohomology
\begin{equation}
f_!:H^{*-r_f}_{\Gamma,deloc}(M) \longrightarrow H^*_{\Gamma,deloc}(N),
\end{equation}	
where $r_f$ is the rank of the vector bundle $T_f=TM\oplus f^*TN$, and under the assumption that $T_f^*$ admits a $spin^c$-structure, in other words, $f$ is $\Gamma$-equivariant and $K$-oriented. The construction of these shriek maps follows the same lines of the shriek maps in $K$-theory constructed in \cite{CaWangAENS} (following the ideas proposed originally in \cite{BC} and \cite{Concg}). Indeed we use $f$ to construct an appropriate deformation groupoid that allows to define $f_!$ as a composition of a Thom isomorphism followed by a deformation morphism, see Section \ref{delocpushforwardsubsection} for more details.

A very important and fundamental point to remark is that to be able to apply constructions similar to those used in $K$-theory one needs, besides suitable groupoids, appropriate smooth subalgebras of the $C^*$-algebras of the corresponding groupoids, in particular appropriate enough to have the right exact sequences that lead to classical arguments in algebraic topology (the homotopy invariance, periodicity, six term exact sequences, etc.). In our case we explain in Appendices \ref{appendixgrpdalgebras} and \ref{appendixB} which are the main groupoids used in this paper and which algebras we are using in each case when applying cohomology (mainly periodic cyclic (co)homology) to them. In Appendix \ref{appendixdefindices} we construct and study the deformation index type morphisms applied to those algebras/groupoids in periodic cyclic (co)homology.

The first main result of this paper is that these cohomological pushforward maps are compatible with the K-theoretical pushforward maps (defined in \cite{CaWangAENS}) via a twisted Chern character. The result states as follows (Theorem \ref{RRthm}):

\begin{theorem}[Delocalised Riemman-Roch]\label{RRthmINTRO}
Let $f:M\longrightarrow N$ be a $\Gamma$-equivariant, $K$-oriented $C^\infty$-map between two $\Gamma$-proper cocompact manifolds $M$ and $N$. The following diagram is commutative
\begin{equation}
\xymatrix{
K^{*-r_f}_{\Gamma}(M)\ar[d]_-{ch_{Td_M^\Gamma}}\ar[r]^-{f_!}& K^*_{\Gamma}(N)\ar[d]^-{ch_{Td_N^\Gamma}} \\
H^{*-r_f}_{\Gamma,deloc}(M)\ar[r]_-{f_!}& H^*_{\Gamma,deloc}(N).
}
\end{equation}
\end{theorem}
In the above statement the twisted Chern character $ch_{Td_M^\Gamma}$ is the morphism
\begin{equation}
ch_{Td_M^\Gamma}:K^*_\Gamma(M)\longrightarrow  H^{*}_{\Gamma,deloc}(M)
\end{equation}
given by $x\mapsto ch^\Gamma_M(x)\wedge Td_M^\Gamma$, where $Td_M^\Gamma$ is a delocalised Todd class, recalled/defined in Section \ref{subsectionRR}.

In order to establish  the above delocalised Riemann-Roch diagram, we need to establish first the functoriality of the pushforward maps in delocalised cohomology. This again is done by using appropriate deformation groupoids, appropriate smooth subalgebras associated to them and classic cohomological properties that allows to adapt to this setting the K-theoretical analog proof. The result (Theorem \ref{wrongwayfunctorialitycohomology}) states as follows:

\begin{theorem}[Pushforward functoriality in delocalised cohomology]\label{wrongwayfunctorialitycohomologyINTRO}
Let $f:M\longrightarrow N$ and $g:N\longrightarrow L$ be $\Gamma$-equivariant, $K$-oriented $C^\infty$-maps between $\Gamma$-proper cocompact manifolds as above. Then
\begin{equation}
(g\circ f)_!=g_!\circ f_!
\end{equation}
as morphisms from $H^*_{\Gamma,deloc}(M)$ to $H^{*+r_{g\circ f}}_{\Gamma,deloc}(L)$ ($r_{g\circ f}=r_f+r_g\,\, mod \, 2$).
\end{theorem}

One of the first things one can do with such a result is to assemble the groups $H^*_{\Gamma,deloc}(M)$ by using the pushforward maps and in an essential way the pushforward functoriality.
For $*=0,1\,\, mod\, 2$,  we can consider the abelian groups
\begin{equation}
H^*_{top}(\Gamma)=\varinjlim_{f_!} H^*_{\Gamma,deloc}(M).
\end{equation}
where the limit is taken over the $\Gamma$-proper cocompact $spin^c$ manifolds of dimension equal to $*$ modulo $2$.

Equivalently it can be defined as the abelian free group with generators $(M,\omega)$ with $M$ a $\Gamma$-proper cocompact $spin^c$ manifold and $\omega \in H^*_{\Gamma,deloc}(M)$ with dimension of $M$ equal to $*$ modulo $2$, and the equivalence relation generated by
$\omega\sim f_!(\omega)$ for $f:M\to N$ as above.

As a direct consequence of the two theorems above together with the pushforward functoriality in $K$-theory (Theorem 4.2 in \cite{CaWangAENS}) we can assemble the twisted Chern characters and get the following  Chern character for any discrete group $\Gamma$  (Theorem \ref{Cherntopthm}):

\begin{theorem}[Chern character for discrete groups]\label{CherntopthmINTRO}
For any discrete group $\Gamma$, there is a well defined Chern character morphism
\begin{equation}
ch^{\Gamma}: K^*_{top}(\Gamma) \longrightarrow H^*_{top}(\Gamma)
\end{equation}
given by
\begin{equation}
ch^{\Gamma}([M,x])=[M, ch^\Gamma_M(x)\wedge Td_M^\Gamma].
\end{equation}
Furthermore, it is an isomorphism once tensoring with $\mathbb{C}$.
\end{theorem}

In \cite{Voigt2} Voigt constructs a Chern character from bivariant equivariant Kasparov theory to the equivariant bivariant cohomology theory defined previously by Baum and Schneider in \cite{BaumSch}. He obtained in particular a Chern character from the analytic left hand side model of Baum and Connes and some delocalised Baum and Schneider cohomology group. In principle his Chern character and ours should coincide for formal reasons (at least it is expected to be like that), it would be very interesting to explore this relation. Another related Chern character was constructed by Matthey in \cite{Matthey} following the related works of L\"uck in \cite{Luck}, in these papers the analytic model is used as well. Even if we do not compare our Chern morphisms with that by Matthey\footnote{It would be very interesting but it is out of the particular subject of the present article.}, we follow in Section~6 the ideas and results of Matthey on the delocalization of the Baum-Connes assembly map. Now, there is however a big difference on the approaches and techniques developed for these in principle different Chern characters but also in the actual computation in terms of cycles, the Chern character in Theorem~\ref{CherntopthmINTRO} is computed explicitly in terms of cycles and of the delocalised Todd classes thanks to the delocalised Riemann-Roch theorem and  the wrong way functorialities in $K$-theory and in equivariant cohomology. This will play a major role in the actual computation of index pairing formulas developed below.

 To describe what follows we need to give more details on the periodic delocalized cohomology groups $H^*_{\Gamma,deloc}(M)$ for  a $\Gamma$-proper cocompact $spin^c$ manifold $M$. In fact, since the action of $\Gamma$ on $M$ is proper, and  for $*=0,1 \,\, mod \,2$, the associated periodic delocalised cohomology groups can be given, as explained in appendix \ref{appendixTuXu}, by
\begin{equation}\label{deloccohomologydecompositionINTRO}
H^*_{\Gamma,deloc}(M):=\bigoplus_{ \langle g \rangle \in \langle \Gamma\rangle^{fin}}\prod_{k=*\,mod\, 2} H^k_c(M_g\rtimes \Gamma_g),
\end{equation}
where $\langle \Gamma\rangle^{fin}$  stands for the set of conjugacy classes of finite order elements (called elliptic elements as well) of $\Gamma$ with  a fixed finite order element $g$  of its conjugacy class,  $\Gamma_g=\{h\in\Gamma: hg=gh\}$ is the centralizer of $g$; and where $M_g\rtimes \Gamma_g$ stands for the action groupoid associated to the canonical action by conjugation of the centralizer $\Gamma_g=\{h\in\Gamma: hg=gh\}$ on the fixed point submanifold $M_g=\{x\in M: x\cdot g=x\}$, and $H^k_c(M_g\rtimes \Gamma_g)$ stands for the compactly supported de Rham groupoid cohomology as reviewed in the next section.
For these groups, we prove in Section~\ref{sectionChernassembly} Lemma~\ref{lemmaintegrationfibers} that the canonical integrations along the fibers of the canonical groupoid projections
\begin{equation}
M_g\rtimes \Gamma_g\longrightarrow \Gamma_g,
\end{equation}
\begin{equation}
\pi_{M_g}!:\Omega_c^n(M_g\times \Gamma_g^p)\longrightarrow \mathbb{C}\Gamma_g^p,
\end{equation}
with $p\in \mathbb{N}$ and $n=dim\, M_g$, given by
\begin{equation}
\pi_{M_g}!(\omega)(\gamma):=\int_{M_g}\omega(m,\gamma)
\end{equation}
for $\gamma\in \Gamma_g^p$ and for  $\omega\in \Omega_c^n(M_g\times \Gamma_g^p)$, induce a well-defined morphism
\begin{equation}
\xymatrix{
H^*_{\Gamma,deloc}(M)\ar[r]^-{\pi_M!}&H_*(\Gamma,F\Gamma),
}
\end{equation}
for every $M$ of dimension $n=*, mod\, 2$, and where
\begin{equation}
H_*(\Gamma,F\Gamma):=\left( \bigoplus_{  \langle g \rangle \in  \langle \Gamma\rangle^{fin}}\bigoplus_{k=*\,mod\, 2}H_k(\Gamma_g;\mathbb{C})\right)
\end{equation}
is the group homology of $\Gamma$ with coefficients in the complex vector space $F\Gamma$ generated by the elliptic elements (See Appendix~C of Matthey~\cite{Matthey} for a discussion of this isomorphism that uses the Shapiro's Lemma).

Section 6 of this paper is mainly dedicated to prove, through several lemmas and propositions of independently interest, the following theorem.

\begin{theorem}\label{thmmuFINTRO}
The maps
\begin{equation}
\xymatrix{
H^*_{\Gamma,deloc}(M)\ar[r]^-{\pi_M!}&H_*(\Gamma,F\Gamma),
}
\end{equation}
induce a well-defined cohomological assembly map
\begin{equation}
\mu_{F\Gamma}: H^*_{top}(\Gamma)\longrightarrow H_*(\Gamma,F\Gamma)
\end{equation}
given by
\begin{equation}
\mu_{F\Gamma}([M,\omega])=\pi_M!(\omega).
\end{equation}
\end{theorem}

There are also other models in the literature for a cohomological assembly map as in Theorem~\ref{thmmuFINTRO}. See for example \cite{CorTar1} and \cite{EngelMar} for more formal constructions. The novelty of our construction (see Section 5) is that it is given explicitly in terms of cohomological pushforward maps and based on the wrong way functoriality. It would be of course interesting to compare the above assembly to the previous ones, since for example in the above cited papers the authors get very interesting properties related to index theory. In our case, the assembly map from Theorem~\ref{thmmuFINTRO} will allow, as explained below, to give index formulas for the pairing between periodic cyclic cohomology and topological $K$-theory for discrete groups.

Now, the periodic cyclic homology groups of $\mathbb{C}\Gamma$ have been computed. Indeed Burghelea showed in \cite{Bur} that there are isomorphisms
\begin{equation}\label{groupalgebracohomology}
HP^*(\mathbb{C}\Gamma)\cong \left( \prod_{  \langle g \rangle \in  \langle \Gamma\rangle^{fin}}\bigoplus_{k=*\,mod\, 2}H^k(\Gamma_g;\mathbb{C})\right)  \bigoplus \prod_{  \langle g \rangle \in  \langle \Gamma\rangle^{\infty}} T^*(g,\mathbb{C})
\end{equation}
 where $\langle \Gamma\rangle^{\infty}$ stands for the set of conjugacy classes of infinite order elements) of $\Gamma$, $g\in \Gamma$ is a fixed element (of infinite order for the second factor) of its conjugacy class; and where $T^*(g,\mathbb{C})$ are some  limit groups depending  on the cohomology groups of $N_g$, the quotient group of the centralizer $\Gamma_g=\{h\in\Gamma: hg=gh\}$ by the subgroup generated by $g$, but that will not be explicitly described in this paper or its sequel since this factor will not enter in our index formula computations. In fact, by Burghelea's work \cite{Bur} there is a morphism
\begin{equation}
\xymatrix{
B:  H_*(\Gamma,F\Gamma)\ar[r] & HP_*(\mathbb{C}\Gamma)
}
\end{equation}
which is an isomorphism onto its image as a direct factor. In particular, we can consider the assembly map
\begin{equation}
\mu_{\Gamma}:   H^*_{top}(\Gamma)\longrightarrow HP_*(\mathbb{C}\Gamma)
\end{equation}
given by the composition of the cohomological assembly map $\mu_{F\Gamma}$ of the theorem above followed by Burghelea's morphism $B$.

We are ready to state the final theorem that combines all the above mentioned results and gives an explicit formula for the index pairing. For this, note that there are canonical morphisms (for every $g$ of finite order)
\begin{equation}
 H^k(\Gamma_g;\mathbb{C})\stackrel{\pi_g^*}{\longrightarrow}H^k(M_g\rtimes \Gamma_g)
\end{equation}
induced from the canonical groupoid projection $\pi_g: M_g\rtimes \Gamma_g\to \Gamma_g$. As we justify in Section~\ref{sectionChernassembly} these pullback morphisms are in duality with the morphisms induced by integration along the fibers described in the last theorem. This finally leads to the following theorem (see Theorems~\ref{Maincorollarygrouphomoogy} and \ref{ChernassemblythmHP}) that  gives an index theoretical geometric formula for the pairing.

\begin{theorem}[The Chern assembly map]\label{ChernassemblythmINTRO}
For any discrete group $\Gamma$, there is well-defined morphism
\begin{equation}
\xymatrix{
K^*_{top}(\Gamma)\ar[rr]^-{ch_\mu^{\Gamma}}&&HP_*(\mathbb{C}\Gamma)\\
}
\end{equation}
given by the composition of the Chern character
\begin{equation}
ch^{\Gamma}: K^*_{top}(\Gamma) \longrightarrow H^*_{top}(\Gamma)
\end{equation}
followed by the cohomological assembly
\begin{equation}
\mu_{\Gamma}: H^*_{top}(\Gamma)\longrightarrow HP_*(\mathbb{C}\Gamma).
\end{equation}
The morphism induces a  well-defined pairing
\begin{equation}
K^*_{top}(\Gamma)\times HP^*(\mathbb{C}\Gamma)\to \mathbb{C}
\end{equation}
computed in every $g$-component (with respect to Burghelea's decomposition above) by
\begin{equation}\label{pairingformulaintro}
\langle [M,x],\tau_g\rangle = \langle ch_M^g(x)\wedge Td_g^M,\pi_g^*(\tau_g)\rangle
\end{equation}
where the right hand side corresponds the pairing between $\Gamma_g$-invariant forms and currents on $M_g$.
The morphism $ch_\mu^{\Gamma}$ will be called the Chern-Baum-Connes assembly map of the group $\Gamma$.
\end{theorem}

We remark  that the Chern-Baum-Connes assembly map in Theorem~\ref{ChernassemblythmINTRO} is explicitly computed in terms of cycles and delocalised Todd classes. It would be interesting to compare our Chern assembly with the ones obtained by other constructions; For example, as observed by Engel and Marcinkowski in \cite{EngelMar}, one can use Yu's algebraic assembly map \cite{Yualgebraic} to get as well a Chern-Baum-Connes assembly map. One can also use the cohomological assembly map studied by Corti\~{n}as and Tartaglia in \cite{CorTar1}. As far as we know there are no comparison results about these, in principle different, models for the Chern assembly map.

Also, relating the $K$-theoretical assembly map with the cohomological assembly map is also part of the program of studying secondary invariants. For example, in \cite{PSZ}, Piazza, Schick and Zenobi used noncommutative de Rham theory and delocalised versions of it to produce some long surgery exact sequences in homology and they obtained very deep results on the study of higher rho invariants by mapping the $K$-theoretical surgery sequence to the homological one. It would be also very interesting to relate our constructions to theirs and study secondary invariants. For example, in \cite{PiaZen}, Piazza and Zenobi used groupoid techniques in a crucial way to study some related problems for singular spaces. Another related work in this direction can be found in \cite{CWXY}.

All the results in this paper are based on the use of classical algebraic topology tools once admitting the use of groupoids and deformation groupoids. As far as we are aware, all the previous partial computations (or particular cases) of index pairings of the kind studied in this paper were computed using mainly the local index formula techniques or related analytical methods. We will discuss elsewhere how the deformation groupoid approach gives a geometrical conceptual explanation and a new insight on the analytical local index techniques.

{\bf Ackowledgements:} We would like to thank Alexander Engel and Christian Voigt for their comments on our article and for pointing out to us very interesting previous works of them and other authors that are closely related to our paper.

\section{Preliminaries}
\subsection{Deformation Lie groupoids}

In this paper we deal with Lie groupoids for which we are going to use the notation $G\rightrightarrows M$, see \cite{Mac,Pat} for more details on Lie groupoids.

In this section we want to recall some particular groupoids that will be used in this paper, their construction/definition is based on the deformation to the normal cone construction to be recalled in the following.

Let $M$ be a $C^\infty$ manifold and $X\subset M$ be a $C^\infty$ submanifold. We denote
by $\mathscr{N}(M,X)$ the normal bundle to $X$ in $M$.
We define the following set
\begin{align}
\mathscr{D}(M,X):= \left( \mathscr{N}(M,X) \times {0} \right) \bigsqcup   \left(M \times \mathbb{R}^* \right).
\end{align}
The purpose of this section is to recall how to define a $C^\infty$-structure on $\mathscr{D}(M,X)$. This is more or less classical, for example
it was extensively used in \cite{HS}.

Let us first consider the case where $M=\mathbb{R}^p\times \mathbb{R}^{n-p}$
and $X=\mathbb{R}^p \times \{ 0\}$ ( here we
identify  $X$ canonically with $ \mathbb{R}^p$). We denote by
$q=n-p$ and by $\mathscr{D}_{p}^{n}$ for $\mathscr{D}(\mathbb{R}^n,\mathbb{R}^p)$ as above. In this case
we   have that $\mathscr{D}_{p}^{n}=\mathbb{R}^p \times \mathbb{R}^q \times \mathbb{R}$ (as a
set). Consider the
bijection  $\psi: \mathbb{R}^p \times \mathbb{R}^q \times \mathbb{R} \rightarrow
\mathscr{D}_{p}^{n}$ given by
\begin{equation}\label{psi}
\psi(x,\xi ,t) = \left\{
\begin{array}{cc}
(x,\xi ,0) &\mbox{ if } t=0 \\
(x,t\xi ,t) &\mbox{ if } t\neq0
\end{array}\right.
\end{equation}
whose  inverse is given explicitly by
$$
\psi^{-1}(x,\xi ,t) = \left\{
\begin{array}{cc}
(x,\xi ,0) &\mbox{ if } t=0 \\
(x,\frac{1}{t}\xi ,t) &\mbox{ if } t\neq0
\end{array}\right.
$$
We can consider the $C^\infty$-structure on $\mathscr{D}_{p}^{n}$
induced by this bijection.

We pass now to the general case. A local chart
$(\mathscr{U},\phi)$ of $M$ at $x$  is said to be an $X$-slice   if
\begin{itemize}
\item[1)]  $\mathscr{U}$  is an open neighborhood of $x$ in $M$ and  $\phi : \mathscr{U}  \rightarrow U \subset \mathbb{R}^p\times \mathbb{R}^q$ is a diffeomorphsim such that $\phi(x) =(0, 0)$.
\item[2)]  Setting $V =U \cap (\mathbb{R}^p \times \{ 0\})$, then
$\phi^{-1}(V) =   \mathscr{U} \cap X$ , denoted by $\mathscr{V}$.
\end{itemize}
With these notations understood, we have $\mathscr{D}(U,V)\subset \mathscr{D}_{p}^{n}$ as an
open subset.   For $x\in \mathscr{V}$ we have $\phi (x)\in \mathbb{R}^p
\times \{0\}$. If we write
$\phi(x)=(\phi_1(x),0)$, then
$$ \phi_1 :\mathscr{V} \rightarrow V \subset \mathbb{R}^p$$
is a diffeomorphism.  Define a function
\begin{equation}\label{phi}
\tilde{\phi}:\mathscr{D}(\mathscr{U},\mathscr{V}) \rightarrow \mathscr{D}(U,V)
\end{equation}
by setting
$\tilde{\phi}(v,\xi ,0)= (\phi_1 (v),d_N\phi_v (\xi ),0)$ and
$\tilde{\phi}(u,t)= (\phi (u),t)$
for $t\neq 0$. Here
$d_N\phi_v: N_v \rightarrow \mathbb{R}^q$ is the normal component of the
 derivative $d_v\phi$ for $v\in \mathscr{V}$. It is clear that $\tilde{\phi}$ is
 also a  bijection. In particular,  it induces a $C^{\infty}$ structure on $\mathscr{D}_{\mathscr{V}}^{\mathscr{U}}$.
Now, let us consider an atlas
$ \{ (\mathscr{U}_{\alpha},\phi_{\alpha}) \}_{\alpha \in \Delta}$ of $M$
 consisting of $X-$slices. Then the collection $ \{ (\mathscr{D}(\mathscr{U}_\alpha,\mathscr{V}_{\alpha}),\tilde{\phi}_{\alpha})
  \} _{\alpha \in \Delta }$ is a $C^\infty$-atlas of
  $\mathscr{D}(M,X)$ (Proposition 3.1 in \cite{Ca4}).

\begin{definition}[Deformation to the normal cone]
Let $X\subset M$ be as above. The set
$\mathscr{D}(M,X)$ equipped with the  $C^{\infty}$ structure
induced by the atlas of  $X$-slices is called
 the deformation to the  normal cone associated  to   the embedding
$X\subset M$.
\end{definition}


One important feature about the deformation to the normal cone is the functoriality. More explicitly,  let
 $f:(M,X)\rightarrow (M',X')$
be a   $C^\infty$ map
$f:M\rightarrow M'$  with $f(X)\subset X'$. Define
$ \mathscr{D}(f): \mathscr{D}(M,X) \rightarrow \mathscr{D}(M',X') $ by the following formulas: \begin{enumerate}
\item[1)] $\mathscr{D}(f) (m ,t)= (f(m),t)$ for $t\neq 0$,

\item[2)]  $\mathscr{D}(f) (x,\xi ,0)= (f(x),d_Nf_x (\xi),0)$,
where $d_Nf_x$ is by definition the map
\[  (\mathscr{N}(M,X))_x
\stackrel{d_Nf_x}{\longrightarrow}  (\mathscr{N}(M',X'))_{f(x)} \]
induced by $ T_xM
\stackrel{df_x}{\longrightarrow}  T_{f(x)}M'$ and $\mathscr{N}(M',X')$ is the normal bundle of $X'$ in $M'$.
\end{enumerate}
 Then $\mathscr{D}(f):\mathscr{D}(M,X) \rightarrow \mathscr{D}(M',X')$ is a $C^\infty$-map (Proposition 3.4 in \cite{Ca4}). In the language of categories, the deformation to the normal cone  construction defines a functor
\begin{equation}\label{fundnc}
\mathscr{D}: \mathscr{C}_2^{\infty}\longrightarrow \mathscr{C}^{\infty} ,
\end{equation}
where $\mathscr{C}^{\infty}$ is the category of $C^\infty$-manifolds and $\mathscr{C}^{\infty}_2$ is the category of pairs of $C^\infty$-manifolds.

We  briefly discuss here the deformation groupoid of an immersion of groupoids which is called  the normal groupoid in   \cite{HS}.

Given  an immersion of Lie groupoids $\gr_1\stackrel{\varphi}{\rightarrow}\gr_2$, let
$\gr^N_1=\mathscr{N}(\gr_2,\gr_1)$ be the total space of the normal bundle to $\varphi$, and $(\gr_{1}^{(0)})^N$ be the total space of the normal bundle to $\varphi_0: \go_1 \to \go_2$.  Consider $\gr^N_1\rightrightarrows (\gr_{1}^{(0)})^N$ with the following structure maps: The source map is the derivation in the normal direction
$d_Ns:\gr^N_1\rightarrow (\gr_{1}^{(0)})^N$ of the source map (seen as a pair of maps) $s:(\gr_2,\gr_1)\rightarrow (\gr_{2}^{(0)},\gr_{1}^{(0)})$ and similarly for the target map.

 The deformation to the normal cone construction allows us to consider a $C^{\infty}$ structure on
$$
\gr_{\varphi}:=\left( \gr^N_1\times \{0\} \right) \bigsqcup  \left( \gr_2\times \mathbb{R}^*\right),
$$
such that $\gr^N_1\times \{0\}$ is a closed saturated submanifold and so $\gr_2\times \mathbb{R}^*$ is an open submanifold.
The following result is  an immediate consequence of the functoriality of the deformation to the normal cone construction.

\begin{proposition}[Hilsum-Skandalis, 3.1, 3.19 \cite{HS}]\label{HSimmer}
Consider an immersion $\gr_1\stackrel{\varphi}{\rightarrow}\gr_2$ as above.
Let $\gr_{\varphi_0}:= \big( (\gr_{1}^{(0)})^N\times \{0\} \big) \bigsqcup \big(  \gr_{2}^{(0)}\times \mathbb{R}^* \big)$ be the deformation to the normal cone of  the  pair $(\gr_{2}^{(0)},\gr_{1}^{(0)})$. The groupoid
\begin{equation}
\gr_{\varphi}\rightrightarrows\gr_{\varphi_0}
\end{equation}
with structure maps compatible  with the ones of the groupoids $\gr_2\rightrightarrows \gr_{2}^{(0)}$ and $\gr_1^N\rightrightarrows (\gr_{1}^{(0)})^N$, is a Lie groupoid with $C^{\infty}$-structures coming from  the deformation to the normal cone.
\end{proposition}

The Lie groupoid above will be called the normal groupoid associated to $\phi$, besides the article of Hilsum and Skandalis in which it appeared for the first time, these kind of groupoids were extensively used in \cite{CaWangAENS}, see also the recent article \cite{Mohsen} for more details on these groupoids.

One of the motivations of these kind of groupoids is to be able to define deformation indices.
Indeed, restricting the deformation to the normal cone construction to the closed interval $[0,1]$ and since the groupoid $\gr_2 \times (0,1]$ is an open saturated subgroupoid of $\gr_{\varphi}$ (see 2.4 in \cite{HS} or \cite{Ren} for more details), we have a short exact sequence of $C^*-$algebras
\begin{equation}\label{segtimm}
0 \rightarrow C^*(\gr_2 \times (0,1]) \longrightarrow C^*(\gr_{\varphi})
\stackrel{ev_0}{\longrightarrow}
C^*(\gr_1^N) \rightarrow 0,
\end{equation}
with $C^*(\gr_2 \times (0,1])$ contractible. Then the 6-term exact sequence in $K$-theory  provides the isomorphism
\[
(ev_0)_{*}:  K_*(C^*(\gr_\varphi))   \cong K_*(C^*(\gr_1^N)).
\]
 Hence
we can define the index morphism
$$D_{\varphi}:K_*(C^*(\gr_1^N))\longrightarrow K_*(C^*(\gr_2))$$
between the K-theories of the maximal $C^*$-algebras as the induced deformation index morphism
\[
D_{\varphi}:=(ev_1)_*\circ(ev_0)_{*}^{-1}: K_*(C^*(\gr_1^N)) \cong K_*(C^*(\gr_\varphi))  \longrightarrow  K_*(C^*(\gr_2)) .
\]

As we will see in the appendices, with the use of appropriate algebras, one can construct similar deformation morphisms in cyclic periodic (co)homology.

Another important example that will be used in this paper is constructed from the example of normal groupoid above and the functoriality of the deformation to the normal cone construction. Indeed, one can for instance consider the classic and most famous example of normal groupoid, the so-called tangent  or Connes groupoid of a manifold $M$, it is indeed the normal groupoid associated to the inclusion
$$M\to M\times M$$
of the diagonal, seen as a unit groupoid, in the pair groupoid of the manifold. Up to a (non canonical) identification of the normal bundle of this inclusion with the tangent bundle $TM$ the Connes tangent groupoid takes the following form
$$G_M^{tan}=TM\bigsqcup (M\times M) \times (0,1]\rightrightarrows M\times [0,1].$$
Now, in the presence of an action of a group $\Gamma$ on $M$ one can consider the associated diagonal action on $M\times M$. Since this action preserves the diagonal it is an immediate consequence of the functoriality of the deformation to the normal cone construction that there is an induced action of $\Gamma$ on the tangent groupoid $G_M^{tan}$ inducing a semi-direct deformation groupoid
$$G_M^{tan}\rtimes \Gamma\rightrightarrows M\times [0,1]$$
explicitly described in the appendices.

\subsection{Lie groupoid cohomology}\label{sectionLiegrpdcohomology}

In this section we will describe some (co)homology groups associated to Lie groupoids and some properties we will use later on the paper. The first three subsections are reviews of classic material, we are going to follow the nice paper/survey by K. Behrend ``Cohomology of Stacks", \cite{Beh}, for more details and proofs the reader can look at the more complete source \cite{BGNX}.

\subsubsection{The de Rham cohomology}

Let $G\rightrightarrows M$ be a Lie groupoid, we denote as usual by $G^{(p)}$ the manifold of $p$-composables arrows. In particular, $M=G^{(0)}$ and $G^{(1)}=G$. For every $p\in \mathbb{N}$ there are $p+1$ structural maps
$$G^{(p)}\stackrel{\partial_i}{\longrightarrow} G^{(p-1)}$$
for $i=0,..,p$, given by:
$$\partial_0(\gamma_1,...,\gamma_p)=(\gamma_2,...,\gamma_p),$$
$$\partial_p(\gamma_1,...,\gamma_p)=(\gamma_1,...,\gamma_{p-1}),$$
and, for $i=1,...,p-1$,
$$\partial_i(\gamma_1,...,\gamma_p)=(\gamma_1,...,\gamma_i\circ \gamma_{i+1},...,\gamma_p).$$
A direct computation shows that $\partial^2=0$ where
$$\partial:\Omega^q(G^{(p-1)})\longrightarrow \Omega^q(G^{(p)})$$
is given by $\partial:=\sum_{i=0,...,p}(-1)^i\partial_i^*$. For a fixed $q$,  $(\Omega^q(G^{(\cdot)}),\partial)$ is then a complex.
The exterior de Rham derivative $d:\Omega^q\to \Omega^{(q+1)}$ connects the various complexes with each other. These lead to a double complex
\begin{equation}
\xymatrix{
\vdots&\vdots&\vdots&\\
\Omega^2(G^{(0)})\ar[r]^-\partial\ar[u]^-d&\Omega^2(G^{(1)})\ar[u]^-d\ar[r]^-\partial&\Omega^2(G^{(2)})\ar[u]^-d\ar[r]^-\partial&\cdots\\
\Omega^1(G^{(0)})\ar[r]^-\partial\ar[u]^-d&\Omega^1(G^{(1)})\ar[u]^-d\ar[r]^-\partial&\Omega^1(G^{(2)})\ar[u]^-d\ar[r]^-\partial&\cdots\\
\Omega^0(G^{(0)})\ar[r]^-\partial\ar[u]^-d&\Omega^0(G^{(1)})\ar[u]^-d\ar[r]^-\partial&\Omega^0(G^{(2)})\ar[u]^-d\ar[r]^-\partial&\cdots
}
\end{equation}
with associated total complex
\begin{equation}
C^n_{dR}(G):=\bigoplus_{p+q=n}\Omega^q(G^{(p)})
\end{equation}
and total differential $\delta:C^n_{dR}(G)\to C^{n+1}_{dR}(G)$ given by
\begin{equation}
\delta(\omega)=\partial(\omega)+(-1)^pd(\omega),
\end{equation}
for $\omega\in \Omega^q(G^{(p)})$.

\begin{definition}[The de Rham cohomology of a Lie groupoid]
The complex $(C^\bullet_{dR}(G),\delta)$ is called the de Rham complex of the Lie groupoid $G\rightrightarrows M.$ Its cohomology groups
\begin{equation}
H^n_{dR}(G):=H^n(C^\bullet_{dR}(G),\delta)
\end{equation}
are called the de Rham cohomology groups of $G\rightrightarrows M$.
\end{definition}

\begin{remark}
As proved in \cite{Beh}, the above groups are invariant under Morita equivalence of Lie groupoids and hence they define groups associated to the corresponding stacks.
\end{remark}
\subsubsection{Cohomology with compact supports}

Let $G\rightrightarrows M$ be a Lie groupoid. Consider the following two numbers
$$r=dim\,G^{(1)}-dim\,G^{(0)}\,\,and\,\, n=2dim\,G^{(0)}-dim\,G^{(1)}=dim\,G^{(0)}-r.$$

Let $\Omega^q_c(G^{(p)})$ be the space of differential forms on $G^{(p)}$ with compact support. Consider the double complex
\begin{equation}\label{dRcomplexcompsupp}
\xymatrix{
\cdots	\ar[r]^-{\partial_!}&\Omega_c^{n+3r}(G^{(2)})\ar[r]^-{\partial_!}&\Omega_c^{n+2r}(G^{(1)})\ar[r]^-{\partial_!}&\Omega_c^{n+r}(G^{(0)})\\
\cdots	\ar[r]^-{\partial_!}&\Omega_c^{n+3r-1}(G^{(2)})\ar[r]^-{\partial_!}\ar[u]^-d&\Omega_c^{n+2r-1}(G^{(1)})\ar[u]^-{-d}\ar[r]^-{\partial_!}&\Omega_c^{n+r-1}(G^{(0)})\ar[u]^-d\\
\cdots	\ar[r]^-{\partial_!}&\Omega_c^{n+3r-2}(G^{(2)})\ar[r]^-{\partial_!}\ar[u]^-d&\Omega_c^{n+2r-2}(G^{(1)})\ar[u]^-{-d}\ar[r]^-{\partial_!}&\Omega_c^{n+r-2}(G^{(0)})\ar[u]^-d\\
&\vdots\ar[u]^-d&\vdots\ar[u]^-{-d}&\vdots \ar[u]^-d
}
\end{equation}
where the vertical differential is the usual exterior derivative $d$ and where the horizontal differential is defined in terms of
\begin{equation}\label{sigma}
\partial_!:=\sum_i(-1)^i(\partial_i)_!
\end{equation}
with $(\partial_i)_!:\Omega_c^{q+r}(G^{(p)})\to \Omega_c^{q}(G^{(p-1)})$ is obtained by integration along the fibers of the structural map $\partial_i$.

In fact, for $\gamma\in \Omega_c^q(G^{(p)})$, the horizontal differential is given by
$$\gamma\mapsto (-1)^p\partial_!\gamma.$$
The associated total complex is given by
\begin{equation}
C^\nu_c(G):=\bigoplus_{j=0,...,\nu}\Omega_c^{n+(j+1)r-\nu+j}(G^{(j)})
\end{equation}
with total differential
\begin{equation}
\delta(\gamma)=\partial_!(\gamma)+(-1)^{p}d(\gamma), \,\,\, for \,\,\, \gamma\in \Omega_c^{q}(G^{(p)}).
\end{equation}

\begin{definition}[The de Rham cohomology with compact supports of a Lie groupoid]
The complex $(C^\bullet_{c}(G),\delta)$ is called the compactly supported de Rham complex of the Lie groupoid $G\rightrightarrows M.$ Its homology groups
\begin{equation}
H^\nu_{c}(G):=H_\nu(C^\bullet_{c}(G),\delta)
\end{equation}
are called the compactly supported de Rham cohomology groups of $G\rightrightarrows M$.
\end{definition}

\subsubsection{Currents and Poincar\'e duality}\label{PDuality}

Given $\gamma\in \Omega^q(G^{(p)})$ and $\omega\in \Omega_c^{pr-q+dim\,G^{(0)}}(G^{(p)})$, the form
\begin{equation}
\gamma \cap \omega \in \Omega^{dim\, G^{(0)}}(G^{(0)})
\end{equation}
defined in \cite{Beh} page 12, allows to define for every such an $\gamma\in \Omega^q(G^{(p)})$ a current
\begin{equation}
C_\gamma: \Omega_c^{pr-q+dim\,G^{(0)}}(G^{(p)})\to \mathbb{R}
\end{equation}
given by
\begin{equation}
C_\gamma(\omega):=\int_{G^{(0)}}\gamma\cap\omega
\end{equation}

\begin{example}\label{examplecapproduct}
Let $G=M\rtimes \Gamma\rightrightarrows M$ be the transformation Lie groupoid associated to the action (by diffeomorphisms) of a discrete countable group $\Gamma$ on a smooth manifold $M$ with $dim\, M=n$. Let $p\in \mathbb{N}$ and let $f\in \mathbb{C}\Gamma^p$. We set $\gamma:=\pi^*f\in\Omega^0_c(M\times \Gamma^p)$ where $\pi:M\times \Gamma^p\to \Gamma^p$ is the projection on the second factor. Let $\omega\in \Omega^n_c(M\times \Gamma^p)$, in this example the form $\gamma\cap \omega\in \Omega^n_c(M)$ is explicitly given by
\begin{equation}
\pi_1!(\omega \wedge f)
\end{equation}
where $\pi_1:M\times \Gamma^p\to M$ is the projection in the first coordinate. In other words, for $m\in M$ we have
\begin{equation}
(\omega \cap \pi^*f)(m)=\sum_{g\in \Gamma^p} f(g)\cdot \omega(m,g).
\end{equation}
We can consider in particular the current defined by $f\in \mathbb{C}\Gamma^p$,
\begin{equation}
C_{f}(\omega):=\int_M \omega \cap \pi^*f.
\end{equation}
\end{example}

Following from Behrend's paper \cite{Beh}, one can check that the above pairing gives a well defined pairing on cohomology
\begin{equation}
H^k_{dR}(G)\times H^{n-k}_c(G)\to \mathbb{R},
\end{equation}
which is perfect (Proposition 19 in \cite{Beh}), hence giving Poincar\'e Duality, under the assumption that both $G^{(1)}$ and $G^{(0)}$ have compatible (via $s$ and $t$) finite good covers. This assumption holds whenever the groupoid is associated to a proper cocompact action.

\subsubsection{The de Rham cohomology with fiberwise rapidly decreasing support}

For some particular groupoids one might consider a slight generalization of the above compactly supported cohomology and pairing. We will detail here a particular example we use in this paper. Let $E\to M$ be a $\Gamma$-equivariant, proper vector bundle. We consider the action groupoid
$$G:=E\rtimes \Gamma \rightrightarrows E.$$
Let $\Omega_{sch}^q(G^{(p)})$ be the space of differential forms on $G^{(p)}\cong E\times \Gamma^p$ with compact horizontal support ({\it i.e.} in the direction of the base $M$) and with a Schwartz condition along the fibers of $E$. The differentials $\partial_!$ of Equation (\ref{sigma}) above are still well defined in $\Omega_{sch}^*(G^{(p)})$. We can hence consider the analog of the de Rham cohomology with compact supports, we call it the de Rham cohomology with Schwartz support
$$H^\nu_{sch}(E\rtimes \Gamma):=H^\nu(C^\bullet_{sch}(E\rtimes \Gamma),\delta).$$



\section{Pushforward maps in periodic delocalised cohomology and delocalised Riemman-Roch theorem}

\subsection{Periodic delocalised cohomology and pushforward maps}\label{delocpushforwardsubsection}
In the case when the action on $\Gamma$ in $M$ is proper, we consider, for $*=0,1 \,\, mod \,2$, the following associated periodic delocalised cohomology groups:
\begin{equation}
H^*_{\Gamma,deloc}(M):=\bigoplus_{\langle \Gamma\rangle^{fin}}\prod_{k=*,\,mod\, 2} H^k_c(M_g\rtimes \Gamma_g),
\end{equation}
where $\langle \Gamma\rangle^{fin}$ stands for the set of conjugacy classes of finite order elements of $\Gamma$, $g\in \Gamma$ is a fixed, finite order element of its conjugacy class, and $M_g\rtimes \Gamma_g$ stands for the action groupoid associated to the action by conjugation of the centralizer $\Gamma_g=\{h\in\Gamma: hg=gh\}$ on the fixed point submanifold $M_g=\{x\in M: x\cdot g=x\}$.

\begin{remark}
In this paper, whenever we deal with the total space $E$ of a $\Gamma$-proper vector bundle we will be using the Schwartz supported version of the de Rham cohomology groups, that is,
\begin{equation}
H^*_{\Gamma,deloc}(E):=\bigoplus_{\langle \Gamma\rangle^{fin}}\prod_{k=*,\,mod\, 2} H^k_{sch}(E_g\rtimes \Gamma_g).
\end{equation}
Here $E_g$ is $E$ restricted to $M_g.$
\end{remark}

Let $f:M\longrightarrow N$ be a $\Gamma$-equivariant $C^\infty$-map between two $\Gamma$-proper cocompact manifolds $M$ and $N$.
We will define below a pushforward map
\begin{equation}
f_!:H^{*-r_f}_{\Gamma,deloc}(M) \longrightarrow H^*_{\Gamma,deloc}(N),
\end{equation}	
where $r_f$ is the rank of the vector bundle $T_f=TM\oplus f^*TN$, and under the assumption that $T_f^*$ admits a $spin^c$-structure, in other words that means $f$ is $\Gamma $ $K$-oriented\footnote{In fact for the cohomological wrong way maps we could require only orientation but we need the $K$-orientation for the analog maps in $K$-theory.}.

One of the first situations in which we already have these pushforward morphisms is in the situation of a $\Gamma$-proper, $\Gamma$-$spin^c$ vector bundle over a $\Gamma$-manifold $M$. Indeed one can in this case consider the zero section map
$$s_0:M\to E$$
and let
$$(s_0)_!:=Th:
H^*_{\Gamma,deloc}(M)\stackrel{Th}{\longrightarrow}H^{*+r_E}_{\Gamma,deloc}(E) $$
as defined in Appendix \ref{appendixThom}, where $r_E$ is the vector bundle rank of $E$.

Before giving the definition of the general pushforward maps let us recall/introduce some spaces and groupoids that will appear in the construction as well as the deformation morphisms in periodic cyclic cohomology that are further studied in the appendix.

Hence, let $f:M\longrightarrow N$ be a $\Gamma$ $C^\infty$-map between two $\Gamma$-proper cocompact manifolds $M$ and $N$, we consider the $\Gamma$-vector bundle
$T_f:=TM\oplus f^*TN$ over $M$ and the deformation groupoid
\begin{equation}
D_f:TM\oplus f^*TN\bigsqcup(M\times M \times N) \times (0,1]
\end{equation}
given as the normal groupoid associated to the $\Gamma$-map
$$M\stackrel{\Delta\times f}{\longrightarrow}M\times M \times N.$$
As reviewed in the appendices, the functoriality of the deformation to the normal cone construction yields a semi-direct product groupoid
\begin{equation}
D_f\rtimes \Gamma \rightrightarrows f^*TN\bigsqcup M\times N\times (0,1]	
\end{equation}
that induces a deformation morphism
\begin{equation}
I_f:HP_*(\mathscr{S}(T_f\rtimes \Gamma))\longrightarrow HP_*(C_c^{\infty}((M\times M \times N)\rtimes \Gamma))
\end{equation}
defined as the composition of the inverse of the morphism induced by the evaluation at zero (as explained in the appendix it is an isomorphism by homotopy invariance of periodic cyclic cohomology)
$$
\xymatrix{
HP_*(\mathscr{S}_c(D_f\rtimes \Gamma))\ar[r]^-{e_0}_-\cong & HP_*(\mathscr{S}(T_f\rtimes \Gamma))
}
$$
followed by the morphism induced by the evaluation at $t=1$
$$HP_*(\mathscr{S}_c(D_f\rtimes \Gamma))\stackrel{(e_1)_*}{\longrightarrow} HP_*(C_c^{\infty}((M\times M \times N)\rtimes \Gamma)).$$

These deformation groupoids and their associated morphisms in $K$-theory (even twisted $K$-theory) appeared already in \cite{CaWangAENS}. For readers' convenience we will give full details in the appendices. We can now define the pushforward maps.

\begin{definition}[Pushward maps in delocalised cohomology with compact supports]\label{defpushforwardCPH}
Let $f:M\longrightarrow N$ be a $\Gamma$-equivariant, $K$-oriented, $C^\infty$-map between two $\Gamma$-proper cocompact manifolds $M$ and $N$. Let $r_f:=r_{T_f}$. We define the morphism
\begin{equation}
f_!:H^{*-r_f}_{\Gamma,deloc}(M) \longrightarrow H^*_{\Gamma,deloc}(N).
\end{equation}	
as the composition of the following morphisms
\begin{itemize}
\item The Thom isomorphism in delocalised cohomology
\begin{equation}
Th:H^{*-r_f}_{\Gamma,deloc}(M)\longrightarrow H^*_{\Gamma,deloc}(T_f^*),
\end{equation}
\item the Tu-Xu inverse isomorphism (explicitly described in \cite{TuXuChern} and recalled in Appendix \ref{appendixTuXu})
\begin{equation}
(TX)^{-1}:	H^*_{\Gamma,deloc}(T_f^*)\longrightarrow HP_*(\mathscr{S}(T_f^*\rtimes \Gamma)),
\end{equation}
\item the isomorphism
\begin{equation}
F: HP_*(\mathscr{S}(T_f^*\rtimes \Gamma))\longrightarrow HP_*(\mathscr{S}(T_f\rtimes \Gamma))
\end{equation}
induced from the Fourier algebra isomorphism,
\item the deformation morphism
\begin{equation}
I_f: HP_*(\mathscr{S}(T_f\rtimes \Gamma))\longrightarrow HP_*(C_c^{\infty}((M\times M \times N)\rtimes \Gamma)),
\end{equation}
\item the isomorphism induced from the Morita equivalence (isomorphism by Proposition 3.6 in \cite{PerrotMorita}) of groupoids $(M\times M \times N)\rtimes \Gamma\sim N\rtimes \Gamma$
\begin{equation}
\mathcal{M}: HP_*(C_c^{\infty}((M\times M \times N)\rtimes \Gamma))\longrightarrow HP_*(C_c^{\infty}(N\rtimes \Gamma)),
\end{equation}
and,
\item the Tu-Xu isomorphism
\begin{equation}
TX: HP_*(C_c^{\infty}(N\rtimes \Gamma))\longrightarrow H^*_{\Gamma,deloc}(N).
\end{equation}
\end{itemize}
\end{definition}

\subsection{Delocalised Riemman-Roch for proper actions}\label{subsectionRR}

As we will see now, one has for these pushforward morphisms the expected compatibility properties with respect to the Chern Character. For $M$ being a  $\Gamma$-proper manifold we consider the $\Gamma$-Chern character morphism

\begin{equation}\label{Cherndelocalised}
ch_M^\Gamma :K^*_{\Gamma}(M)\longrightarrow  H^*_{\Gamma,deloc}(M).
\end{equation}
given explicitly as the composition of the Connes-Chern character\footnote{Defined a priori in $K_*(C_c^\infty(M\rtimes \Gamma))$ but $C_c^\infty(M\rtimes \Gamma)\subset C^*(M\rtimes \Gamma)$ is stable under holomorphic calculus since the action is proper.}
\begin{equation}
Ch :K^*_{\Gamma}(M)\longrightarrow  HP_*(C_c^\infty(M\rtimes \Gamma)).
\end{equation}
followed by the Tu-Xu isomorphism (explicitly described in \cite{TuXuChern} and recalled in Appendix \ref{appendixTuXu})
\begin{equation}
\xymatrix{
HP_*(C_c^\infty(M\rtimes \Gamma))\ar[r]^-{TX}_-{\cong}& H^*_{\Gamma,deloc}(M).}
\end{equation}
In fact, Tu and Xu proved in \cite{TuXuChern} that the above Chern character (\ref{Cherndelocalised}) is an isomorphism up to tensoring with $\mathbb{C}$.

In particular for $M$ a $\Gamma$-proper manifold as above we have a component decomposition
\begin{equation}
ch_M^\Gamma(x)=(ch_M^g(x))_{g\in \langle \Gamma\rangle^{fin}}\in \bigoplus_{g\in \langle \Gamma\rangle^{fin}}\left(\bigoplus_{k=*,\,mod\, 2} H^k_c(M_g\rtimes \Gamma_g)\right)
\end{equation}
for every $x\in K^*_{\Gamma}(M)$.

For $E\longrightarrow M$  a $\Gamma$-proper $spin^c$ vector bundle over a $\Gamma$-proper manifold $M$ we can consider the Thom isomorphism in equivariant $K$-theory (as classically constructed by Kasparov, also reviewed in the appendix in \cite{CaWangAENS})
$$K^*_\Gamma(M)\longrightarrow K^{*+r_E}_\Gamma(E),$$
and for every $g\in \Gamma$ of finite order the delocalised Thom isomorphism (notice $r_E=r_{E_g}$ where $E_g$ is $E$ restricted to $M_g$)
$$Th_g:\bigoplus_{k=*,\,mod\, 2} H^k_c(M_g\rtimes \Gamma_g)\longrightarrow \bigoplus_{k=*,\,mod\, 2} H^{k+r_E}_{sch}(E_g\rtimes \Gamma_g),$$
given as the collection of Thom isomorphisms in each component $$H^k_c(M_g\rtimes \Gamma_g)\to H^{k+r_E}_{sch}(E_g\rtimes \Gamma_g)$$ (see Appendix E or \cite{BGNX} Chapter 4 for more details on these Thom isomorphisms).
We have then the delocalised $Todd$-class
$$Td_E^\Gamma:=(Td_g^E)_{g\in  \langle \Gamma\rangle^{fin}}\in \bigoplus_{g\in \langle G\rangle^{fin}}\left(\bigoplus_{k=*,\,mod\, 2} H^{k}(E_g\rtimes \Gamma_g)\right)$$
where each class $Td_g^E$ is as classic the unique class satisfying
\begin{equation}
Th_g(ch_M^g(x)\wedge Td_g^E)=ch_E^g(Th(x))
\end{equation}
for every $x\in K^*_{\Gamma}(M)$. See \cite{BCfete} Section 9 for more details and a precise definition of these classes.

We can now consider the twisted  $\Gamma$-Chern character morphism associated to $E$

\begin{equation}
ch_{Td_E^\Gamma} :K^*_{\Gamma}(M)\longrightarrow  H^*_{\Gamma,deloc}(M),
\end{equation}
given by
$$ch_{Td_E^\Gamma}(x):=ch_M^\Gamma(x)\wedge Td_E^\Gamma.$$

By construction of the delocalised Todd classes we immediately obtain the following result.

\begin{proposition}\label{RRThomthm}
Let $E\longrightarrow M$ be a $\Gamma$-proper $spin^c$ vector bundle over a $\Gamma$-proper cocompact  manifold $M$. The following diagram is commutative
\begin{equation}
\xymatrix{
K^*_{\Gamma}(M)\ar[d]_-{ch_{Td_E^\Gamma}}\ar[rr]^-{Th}_-\cong&&K^{*+r_E}_{\Gamma}(E)\ar[d]_-{ch_E^\Gamma}\\
 H^*_{\Gamma,deloc}(M)\ar[rr]_-{Th}^-\cong&&H^{*+r_E}_{\Gamma,deloc}(E)
}
\end{equation}
where the morphisms denoted by $Th$ stand for the respective Thom isomorphisms.
\end{proposition}


{\bf Notation:} In the case where $M$ is a $\Gamma$-proper cocompact $spin^c$ manifold $M$ we denote by $Td_M^\Gamma$ the class associated to the cotangent space $T^*M$.


\vspace{2mm}

We can state and prove the compatibility of the above-defined pushforward maps and the delocalised twisted Chern characters. It is one of the main results in this paper.

\begin{theorem}[Delocalised Riemman-Roch]\label{RRthm}
Let $f:M\longrightarrow N$ be a $\Gamma$-equivariant, $K$-oriented, $C^\infty$-map between two $\Gamma$-proper cocompact manifolds $M$ and $N$. The following diagram is commutative
\begin{equation}
\xymatrix{
K^{*-r_f}_{\Gamma}(M)\ar[d]_-{ch_{Td_M^\Gamma}}\ar[r]^-{f_!}& K^*_{\Gamma}(N)\ar[d]^-{ch_{Td_N^\Gamma}} \\
H^{*-r_f}_{\Gamma,deloc}(M)\ar[r]_-{f_!}& H^*_{\Gamma,deloc}(N).
}
\end{equation}
\end{theorem}

\begin{proof}
By definition of the $Todd$-classes we are led to prove the commutativity of the following big diagram
\[
\xymatrix{
K^{*-r_f}_{\Gamma}(M)\ar@/_8pc/[dddddddddddddd]_-{f_!} \ar@{}[rrdd]_-{{\bf I}}\ar[dd]_-{Th}^-\cong \ar[rr]^-{ch_{Td_{T_f^*}^\Gamma}}&& H^{*-r_f}_{\Gamma,deloc}(M)\ar[dd]^-{Th}_-\cong \ar@/^8pc/[dddddddddddddd]^-{f_!}\\
&&\\
K^*_{\Gamma}(T_f^*) \ar[dd]_-{Id}^-\cong\ar[rr]^-{ch_{T_f^*}^\Gamma}\ar@{}[rrdd]_-{{\bf II}}&&H^*_{\Gamma,deloc}(T_f^*)\ar[dd]^-{(TX)^{-1}}_-\cong\\
&&\\
K^*_{\Gamma}(T_f^*) \ar[dd]_-{F}^-\cong\ar[rr]^-{Ch}\ar@{}[rrdd]_-{{\bf III}}&&HP_*(\mathscr{S}(T_f^*\rtimes \Gamma))\ar[dd]^-{F}_-\cong\\
&&\\
K^*_{\Gamma}(T_f) \ar[rr]^-{Ch}\ar@{}[rrdd]_-{{\bf IV}}&&HP_*(\mathscr{S}(T_f\rtimes \Gamma))\\
&&\\
K^*_{\Gamma}(D_f) \ar[uu]^-{(e_0)_*}_-\cong \ar[dd]_-{(e_1)_*} \ar[rr]^-{Ch}\ar@{}[rrdd]_-{{\bf V}}&&HP_*(\mathscr{S}_c(D_f\rtimes \Gamma))\ar[dd]^-{(e_1)_*}\ar[uu]_-{(e_0)_*}^-\cong \\
&&\\
K^*_{\Gamma}(M\times M\times N)\ar[dd]_-{\mathcal{M}}^-\cong \ar[rr]^-{Ch}\ar@{}[rrdd]_-{{\bf VI}}&&HP_*(C_c^\infty((M\times M\times N)\rtimes \Gamma))\ar[dd]^-{\mathcal{M}}_-\cong\\
&&\\
K^*_{\Gamma}(N) \ar[dd]_-{Id}^-\cong\ar[rr]^-{Ch}\ar@{}[rrdd]_-{{\bf VII}}&&HP_*(C_c^\infty(N)\rtimes \Gamma))\ar[dd]^-{TX}_-\cong\\
&&\\
K^*_{\Gamma}(N) \ar[rr]_-{ch_{N}^\Gamma}&&H^*_{\Gamma,deloc}(N)
}
\]
for which it will be enough to justify the commutativity of diagrams {\bf I-VII} above. Remember that the Connes-Chern character is natural, meaning that it commutes with morphisms induced from algebra morphisms, in particular we have then the commutativity of diagrams {\bf III-V}. The commutavity of diagrams {\bf II} and {\bf VII} is trivial, they follow from definition. Commutavity of diagram {\bf I} is precisely Proposition \ref{RRThomthm} above. Finally, we will justify the commutativity of diagram {\bf VI} which is a particular case of naturality of Connes-Chern character with respect to Morita isomorphisms. The Lie groupoid Morita equivalence
\begin{equation}
\xymatrix{
((M\times M)\times N)\rtimes \Gamma\ar[r]^-\sim& N\rtimes \Gamma
}
\end{equation}
induces the trace like function ({\it i.e.} $\tau(k*k')=\tau(k'*k)$)
\begin{equation}
\tau: C_c^\infty((M\times M\times N)\rtimes \Gamma))\longrightarrow C_c^\infty(N\rtimes \Gamma)
\end{equation}
given by
\begin{equation}
\tau(k)(n,g)=\int_Mk(m,m\cdot g, n, g)dm,
\end{equation}
and that induces the Morita equivalence of the correspondent $C^*$-algebras, the isomorphism $\mathcal{M}$ in $K$-theory and the isomorphism $\mathcal{M}$  (isomorphism by Proposition 3.6 in \cite{PerrotMorita}) at the level of $HP_*$. By definition of the Chern-Connes character (see for example \cite{Khal} Chapter 4 or \cite{Concg}), defined at the level of the smooth subalgebras $C_c^\infty((M\times M\times N)$ and $C_c^\infty(N\rtimes \Gamma)$, it is immediate that the diagram commutes, since $\tau$ is a trace.
\end{proof}

\section{Delocalised cohomology wrong way functoriality, Chern character for discrete groups}

\subsection{Delocalised cohomology wrong way functoriality}

As in the case of $K$-theory groups we can prove that the above-defined pushforward morphisms is functorial. This is one of our main results, stated as follows:

\begin{theorem}[Pushforward functoriality in delocalised cohomology]\label{wrongwayfunctorialitycohomology}
Let $f:M\longrightarrow N$ and $g:N\longrightarrow L$ be $\Gamma$ $K$-oriented $C^\infty$-maps between $\Gamma$-proper cocompact manifolds.. Then
\begin{equation}
(g\circ f)_!=g_!\circ f_!
\end{equation}
as morphisms from $H^*_{\Gamma,deloc}(M)$ to $H^{*+r_{g\circ f}}_{\Gamma,deloc}(L)$ ($r_{g\circ f}=r_f+r_g\,\, mod \, 2$).
\end{theorem}
\begin{proof}
For simplicity suppose $M$, $N$ and $L$ have the same dimension, otherwise we only have to add the shifting on the degrees as in the previous sections.

Given $f$ as in the statement, consider the morphism
\begin{equation}
f_*:H^*_{\Gamma,deloc}(M)\to H^*_{\Gamma,deloc}(N)
\end{equation}
given as the composition of
\begin{equation}
{\tiny
\xymatrix{
H^*_{\Gamma,deloc}(M)\ar[r]^-{(TX)^{-1}}&HP_*(C_c^{\infty}(M\rtimes \Gamma))\ar[r]^-{Th}&HP_*(\mathscr{S}(T_f^*\rtimes\Gamma))\ar[r]^-F&HP_*(\mathscr{S}(T_f\rtimes \Gamma))\\
\ar[r]^-{I_f}&HP_*(C_c^{\infty}((M\times M\times N)\rtimes\Gamma))\ar[r]^-{\mathcal{M}}&HP_*(C_c^{\infty}(N\rtimes \Gamma))\ar[r]^-{TX}& H^*_{\Gamma,deloc}(N)
}
}
\end{equation}
By Proposition \ref{RRThomthm} and the fact that the Tu and Xu Chern character (\ref{Cherndelocalised}) is an isomorphism once tensoring with
$\mathbb{C}$, we have that
\begin{equation}\label{f*=f!}
f_*(\cdot\wedge Td_{T_f^*}^\Gamma)=f_!.
\end{equation}
We briefly explain the last equation. The difference between $f_!$ and $f_*$ is in the first two maps used to define them, indeed $f_!$ and $f_*$ would be the equal if the diagram of isomorphisms
\begin{equation}
\xymatrix{
H^*_{\Gamma,deloc}(M)\ar[r]^-{Th}&H^*_{\Gamma,deloc}(T^*_f)\\
HP_*(C_c^\infty(M\rtimes \Gamma))\ar[u]^-{TX}\ar[r]_-{Th}&HP_*(\mathscr{S}(T^*_f\rtimes \Gamma))\ar[u]_-{TX}
}
\end{equation}
was commutative, where, as explained in Appendix \ref{appendixThom}, the Thom isomorphism in cyclic periodic homology above is considered to be in this paper as the unique isomorphism commuting (up to tensoring with $\mathbb{C}$) with the corresponding Thom isomorphism in $K$-theory up to the Chern-Connes character. Hence, by Proposition 3.3 and because the Tu and Xu Chern character (\ref{Cherndelocalised}) is an isomorphism once tensoring with
$\mathbb{C}$, we obtain that the above diagram is not commutative but that the failure to the commutativity is precisely the wedge product with the delocalised Todd class $Td_{T_f^*}$. Hence equation (\ref{f*=f!}) holds.

Now, because of the classical properties for the Todd classes we have
\begin{equation}
Td_{T_f^*}^\Gamma\wedge f^*Td_{T_g^*}^\Gamma=Td_{T_{g\circ f}^*}^\Gamma,
\end{equation}
and hence it is enough to prove the functoriality for the morphisms $f_*$.

Now, the pushforward map $f_*$ decomposes by definition as $(TX)\circ f_*^{HP}\circ(TX)^{-1}$ where
$$f_*^{HP}: HP_*(C_c^\infty(M\rtimes \Gamma))\longrightarrow  HP_*(C_c^\infty(N\rtimes \Gamma)),$$ and an analogous decomposition for $g_*$ can be considered. To conclude the proof it is then enough to show that
\begin{equation}
(g\circ f)_*^{HP}=g_*^{HP}\circ f_*^{HP}.
\end{equation}
The above pushforward functoriality is proved by following the same exact steps as the $K$-theory functoriality Theorem~4.2 in \cite{CaWangAENS} by replacing the $K$-theory functor with the periodic cyclic homology groups and by replacing the $C^*$-algebras with the smooth algebras considered in the appendices. Indeed, notice that proof of Theorem 4.2 in \cite{CaWangAENS} is based on the construction of appropriate groupoids (the same exact groupoids work in our present case) and cohomological properties of the $K$-theory groups that hold as well in periodic cyclic (co)homology, in particular naturality, Morita invariance, Thom isomorphisms and homotopy invariance.
\end{proof}

The above wrong way functoriality theorem allows us to consider in particular the following assembled group.

\begin{definition}[Delocalised topological cohomology for discrete groups]
For a discrete group $\Gamma$ and for $*=0,1\,\, mod\, 2$ we can consider the abelian group
\begin{equation}
H^*_{top}(\Gamma)=\varinjlim_{f_!} H^*_{\Gamma,deloc}(M).
\end{equation}
where the limit is taken over the $\Gamma$ proper cocompact $spin^c$ manifolds of dimension equal to $*$ modulo $2$.

Equivalently it can be defined as the free abelian group with generators $(M,\omega)$ with $M$ a $\Gamma$ proper cocompact $spin^c$ manifold and $\omega \in H^*_{\Gamma,deloc}(M)$, and the equivalence relation generated by $\omega\sim f_!(\omega)$ for $f:M\to N$ as above.
\end{definition}

\begin{remark}
Because the pushforward maps $f_!$  and $f_*^{HP}$ (Notation in proof of Theorem \ref{wrongwayfunctorialitycohomology}) differ by a Tu-Xu isomorphism it is immediate that
\begin{equation}
H^*_{top}(\Gamma)\cong \varinjlim_{f_*^{HP}} HP_*(C_c^{\infty}(M\rtimes \Gamma)),
\end{equation}
where the limit is taken over the $\Gamma$ proper cocompact $spin^c$ manifolds of dimension equal to $*$ modulo $2$.
As we will see in the next section, there is a natural way to map this group into the periodic cyclic group of the discrete group $\Gamma$.
\end{remark}

\subsection{The topological $K$-theory for discrete groups and the Chern character for discrete groups}

The definition of the delocalised topological cohomology for discrete groups follows the same pattern of the definition of the topological $K$-theory for discrete groups as originally defined by Baum and Connes. In this section we will recall the definition of the group $K_{top}^*(\Gamma)$. This group is the original ``Left Hand" side of the Baum-Connes assembly map, it appears first in the papers by Baum and Connes \cite{BCfete} and \cite{BC}. Now, in those papers the well definedness of the groups was sketched and not proved. The fact that it is well defined is a consequence of the wrong way funtoriality theorem proved in a more general context in \cite{CaWangAENS}. Indeed, in ref.cit., the authors constructed the pushforward map (with a possible shifting on the degree depending on the rank of the vector bundle $T_f$ as above)
\begin{equation}\label{pushforwardKtheory}
f_!:K_*(C_r^*(M\rtimes \Gamma))\longrightarrow K_*(C_r^*(N\rtimes \Gamma))
\end{equation}
associated to a $\Gamma$-equivariant $K-$oriented map $f:M\to N$ under the hypothesis that $M$ is a $\Gamma$-proper manifold. The main theorem of \cite{CaWangAENS}, Theorem 4.2, states that this construction is functorial, meaning that
\begin{equation}
(g\circ f)_!=g_!\circ f_!
\end{equation}
for $\Gamma$-equivariant maps $f$ and $g$. One of the main consequences of this wrong way functoriality is that it allows to define the following groups.

\begin{definition}[Topological $K$-theory for discrete groups]
For a discrete group $\Gamma$ and for $*=0,1\,\, mod\, 2$ we can consider the abelian group
\begin{equation}
K^*_{top}(\Gamma)=\varinjlim_{f_!} K_*(C_r^*(M\rtimes \Gamma)),
\end{equation}
where the limit is taken over the $\Gamma$ proper cocompact $spin^c$ manifolds of dimension equal to $*$ modulo $2$.

Equivalently the above group can be defined as the free abelian group with generators $(M,x)$ with $M$ a $\Gamma$ proper cocompact $spin^c$ manifold and $x\in K_*(C_r^*(M\rtimes \Gamma))$, and the equivalence relation generated by $x\sim f_!(x)$ for $f:M\to N$ as above.
\end{definition}

Now, since in the construction of $f_!$ in (\ref{pushforwardKtheory}) the manifold $N$ does not necessarily require to be proper and in particular it can be a point if $M$ is taken with a $\Gamma-spin^c$ structure, the pushforward functoriality theorem in \cite{CaWangAENS} also implies that the Baum-Connes assmebly map
\begin{equation}
K^*_{top}(\Gamma)\stackrel{\mu}{\longrightarrow} K_*(C^*_r(\Gamma))
\end{equation}
given by $\mu([M,x])=(\pi_M)!(x)$ is well defined, where $\pi_M:M\to \{pt\}$.

The constructions and proofs in \cite{CaWangAENS} were based in an essential way on the use of appropriate deformation groupoids and the use of classical (co)homological properties of $K$-theory such as homotopy invariance, naturality of Thom isomorphisms, Morita invariance, etc.

The following theorem is an immediate corollary of the delocalised Riemann-Roch Theorem \ref{RRthm} and of the push-forward functoriality theorems in $K$-theory (Theorem 4.2 in \cite{CaWangAENS}) and in periodic cyclic homology, Theorem \ref{wrongwayfunctorialitycohomology} above.

\begin{theorem}\label{Cherntopthm}
For any discrete group $\Gamma$, there is a well-defined Chern character morphism
\begin{equation}
ch^{\Gamma}: K^*_{top}(\Gamma) \longrightarrow H^*_{top}(\Gamma)
\end{equation}
given by
\begin{equation}
ch^{\Gamma}([M,x])=[M, ch^\Gamma_M(x)\wedge Td_M^\Gamma].
\end{equation}
Furthermore, it is an isomorphism once tensoring with $\mathbb{C}$.
\end{theorem}

\section{Comparison isomorphisms with the ``official" Baum-Connes left hand side models}

In this section we will prove that there are group isomorphisms
\begin{equation}
K^*_{top}(\Gamma)\cong K^\Gamma_*(\underline{E\Gamma}) \,\, \text{and}\,\, H^*_{top}(\Gamma)\cong H^\Gamma_*(\underline{E\Gamma}).
\end{equation}

This will justify that our groups are the good ones to unify several aspects and approaches for index theory in this context, also it will justify that the Chern character in Theorem~\ref{Cherntopthm} is the appropriate one because it is defined on the appropriate groups and explicitly computed in terms of cycles and it makes appear explicitly the delocalized version of Chern and Todd classes.

The first ingredient we need to recall is the classical duality between $K$-homology and $K$-theory. Indeed, for a $\Gamma$-proper cocompact $spin^c$-manifold $M$ there is a Poincar\'e Duality isomorphism between the $\Gamma$-equivariant $K$-homology of $M$ and the $\Gamma$-equivariant $K-$theory of $M$
\begin{equation}\label{PDproperactions}
\xymatrix{
K_*^\Gamma(M)\ar[r]^-{PD_M}_-\cong & K^*_\Gamma(M)
}
\end{equation}
that we will briefly recall; For more details on this duality see Section~4 of \cite{Kasparov88} or \cite{EM2010} for a more general discusion on the subject. In fact, we shall give a version of this duality in relation to the pushforward map in $K$-theory, following closely the work of Baum, Higson and Schick (see particularly Sections~7 and~8 in \cite{BHS}), where a geometric description of the equivariant $K$-homology groups is provided. Indeed, an element in $K_*^\Gamma(M)$ can be represented by the class (the equivalence relation is an obvious notion of bordism) of a triple $(N,f,x)$, where $N$ is a $\Gamma$-proper cocompact $spin^c$-manifold, $f:N\to M$ is a $\Gamma$-equivariant smooth submersion, and $x\in K^*_\Gamma(N)$. One can show directly, thanks to the results of Baum, Higson and Schick, the following classical result.

\begin{proposition}
\label{prop:KPD}
Let $M$ be a $\Gamma$-proper cocompact $spin^c$-manifold. The Poincar\'e Duality isomorphism between the $\Gamma$-equivariant $K$-homology of $M$ and the $\Gamma$-equivariant $K$-theory of $M$
\begin{equation}
\xymatrix{
K_*^\Gamma(M)\ar[r]^-{PD_M}_-\cong & K^*_\Gamma(M)
}
\end{equation}
is given by
\begin{equation}
PD_M([N,f,x])=f_!(x),
\end{equation}
for $[N,f,x]\in K_*^\Gamma(M)$ represented by a triple as above.
\end{proposition}

\begin{proof}
Given $f:N\to M$, a $\Gamma$-equivariant smooth submersion (between $\Gamma$-proper cocompact $spin^c$ manifolds), we consider the $\Gamma$-family index morphism
\begin{equation}
\xymatrix{
K^*_\Gamma(T^*N)\ar[r]^-{Ind_f} & K^*_\Gamma(M)
}
\end{equation}
constructed as classically (with the classical pseudodifferential extensions) by the transverse pseudodifferential calculus developed for example in \cite{HS} or in \cite{Kasparov16} (see also \cite{PPTtransverse} which covers this calculus as well). Similar to the non-equivariant case, the tangent groupoid technique can be applied to show that the following diagram commutes:
\begin{equation}
\xymatrix{
K^*_\Gamma(T^*N)\ar[r]^-{Ind_f} & K^*_\Gamma(M)\\
K^*_\Gamma(N)\ar[ru]_-{f_!}\ar[u]^-{Th}_-\cong &
}.
\end{equation}
Recall that under the Baum-Higson-Schick description we have that (see \cite{Kasparov16} for a $KK$-description)
$$PD_M([N,f,x])=Ind_f(\sigma(x\otimes [D_N)]),$$
where $\sigma: K_*^\Gamma(N)\to K^*_\Gamma(N)$ is the principal symbol class morphism. To conclude, we only need to observe that the morphism
$$K^*_\Gamma(N)\to K^*_\Gamma(T^*N)$$
given by
$$x\mapsto \sigma(x\otimes [D_N])$$
coincides exactly with the Thom isomorphism associated to the $\Gamma$-$spin^c$ structure on $N$.
\end{proof}

As an immediate corollary of Proposition~\ref{prop:KPD} and of the pushforward funtoriality in $K$-theory we obtain the following statement.

\begin{proposition}[Poincar\'e Duality vs. Pushforward]\label{PDproposition}
Let $g:M\to M'$ be a $\Gamma$-equivariant $C^\infty$-map between $\Gamma$-proper cocompact $spin^c$ manifolds. The following diagram commutes:
\begin{equation}
\xymatrix{
K_*^\Gamma(M)\ar[r]^-{g_*}\ar[d]_-{PD_M}^-\cong & K_*^\Gamma(M')\ar[d]^-{PD_{M'}}_-\cong\\
K^*_\Gamma(M)\ar[r]_-{g!}&K^*_\Gamma(M').
}
\end{equation}
\end{proposition}

By the universal property of the classifying space and the work of Baum, Higson and Schick (see last section in \cite{BHS}) we observe that the canonical morphisms
\begin{equation}\label{BHSisos}
\xymatrix{
\varinjlim_{f_*}K_*^\Gamma(M)\ar[r] & \varinjlim_{X\subset \underline{E\Gamma}}K_*^\Gamma(X) \ar[r]& K_*^\Gamma(\underline{E\Gamma})
}
\end{equation}
where the first limit is taken over the $\Gamma$-equivariant $K$-homology groups of $\Gamma$-proper cocompact $spin^c$ manifolds of dimension equal to $*$ modulo $2$ and where the second limit is taken over the $\Gamma$-equivariant K-homology groups of $\Gamma$-proper finite $CW$-complexes, are isomorphisms.

As an immediate corollary of the observation above and of Proposition~\ref{PDproposition}, we obtain the following result.

\begin{corollary}\label{ConnesvsLuck}
Let $\Gamma$ be a countable discrete group. There is a unique isomorphism
\begin{equation}
\xymatrix{
K^*_{top}(\Gamma)\ar[r]^-\lambda_-\cong & K_*^\Gamma(\underline{E\Gamma})
}
\end{equation}
fitting in the following commutative diagram of group isomorphisms
\begin{equation}
\xymatrix{
\varinjlim_{f_*}K_*^\Gamma(M)\ar[r]^-\cong \ar[d]_-{PD}^-\cong& \varinjlim_{X\subset \underline{E\Gamma}}K_*^\Gamma(X) \ar[d]^-\cong \\
K^*_{top}(\Gamma)\ar[r]_-\lambda^-\cong& K_*^\Gamma(\underline{E\Gamma})
}
\end{equation}
where $PD$ is the isomorphism induced from Proposition \ref{PDproposition}.
\end{corollary}


We will now use the results above together with Theorem \ref{Cherntopthm} to show that the delocalized cohomology for discrete groups can be computed as well by the equivariant homology of the classifying space for proper actions. The first thing is that there is a concrete isomorphism
\begin{equation}\label{Mattheysiso}
\xymatrix{
\bigoplus_{g\in \langle\Gamma\rangle^{fin}} K_*(B\Gamma_g)\otimes \mathbb{C}\ar[rr]^-{\phi_*^\Gamma}_-\cong &&K_*^\Gamma(\underline{E\Gamma})\otimes \mathbb{C}
}
\end{equation}
constructed by Matthey; See Section~7 and Theorem~1.3 in~\cite{Matthey}. In fact, Matthey showed how his morphism compares with Luck's Chern
character (\cite{Luck}), showing in particular that they can be intertwined by some precise isomorphisms and so Matthey's morphism is an
isomorphism because L\"uck's Chern character is; See Theorem~0.7 in~\cite{Luck}. In \cite{Luck} and \cite{Matthey} the authors could handle more general cases where coefficients belong to a proper subset of the complex numbers, which will not be needed here. Before coming back to the precise definition of the isomorphism (\ref{Mattheysiso}) in Section~\ref{sec:Chern-Baum-Connes}, note that we do have an isomorphism
\begin{equation}
\label{MattheysisoH}
H_*(\Gamma,F\Gamma)\cong H^*_{top}(\Gamma)
\end{equation}
where $H_*(\Gamma,F\Gamma)$ is the group homology of $\Gamma$ with coefficients in the complex vector space $F\Gamma$ generated by the elliptic elements. The isomorphism (\ref{MattheysisoH}) follows from (\ref{Mattheysiso}) because the left hand side of (\ref{Mattheysiso}) is isomorphic to $H_*(\Gamma,F\Gamma)$, by using for example the classic Chern character in K-homology, and also because the right hand side of (\ref{Mattheysiso}) is isomorphic to $H^*_{top}(\Gamma)$ by putting together Corollary \ref{ConnesvsLuck} and Theorem \ref{Cherntopthm}.

\section{The Chern-Baum-Connes assembly map and the Chern-Connes pairing formula for discrete groups}\label{sectionChernassembly}




\subsection{The Chern-Baum-Connes assembly map in periodic group homology and the pairing formula for discrete groups}
\label{sec:Chern-Baum-Connes}

We already have an explicit formula for the Chern character
\begin{equation}
ch^\Gamma: K^*_{top}(\Gamma) \longrightarrow H^*_{top}(\Gamma)
\end{equation}
given in Theorem \ref{Cherntopthm}. In the following, we propose an explicit geometric description of the cohomological assembly map
\begin{equation}
\mu_{\Gamma}: H^*_{top}(\Gamma)\longrightarrow HP_*(\mathbb{C}\Gamma),
\end{equation}
that will lead, by composing with the above Chern character, to explicit index formulas.
Let us start by denoting, for $*=0,1$,
\begin{equation}
\label{eq:HFGamma}
H_*(\Gamma,F\Gamma):=\left( \bigoplus_{\langle \Gamma\rangle^{fin}}\bigoplus_{k=*,\,mod\, 2}H_k(\Gamma_g;\mathbb{C})\right).
\end{equation}
In principle, in the literature, the left hand side of (\ref{eq:HFGamma}) denotes the group homology of $\Gamma$ with coefficients in the complex vector space $F\Gamma$ generated by the elliptic elements, but this group homology is isomorphic to the right hand side of (\ref{eq:HFGamma}); See Matthey \cite{Matthey} Appendix~C for a discussion of this isomorphism using the Shapiro's Lemma. From now on we shall regard~(\ref{eq:HFGamma}) as a definition for $H_*(\Gamma,F\Gamma)$.

Then, as we recalled in the introduction, the periodic cyclic homology $HP_*(\mathbb{C}\Gamma)$ decomposes into the direct sum of two factors, one of which is isomorphic to the group $H_*(\Gamma,F\Gamma)$ in (\ref{eq:HFGamma}). In fact, we will recall and use below an isomorphism (denoted $B$ as Burghelea) onto its image
\begin{equation}
B:H_*(\Gamma,F\Gamma)\longrightarrow HP_*(\mathbb{C}\Gamma),
\end{equation}
as the above mentioned direct factor. As starting point we will present in the next theorem a geometric assembly map
\begin{equation}
\mu_{F\Gamma}: H^*_{top}(\Gamma)\longrightarrow H_*(\Gamma,F\Gamma)
\end{equation}
explicitly formulated in terms of integration along the fibers of compactly supported forms.

Let us explain this in detail. Let $M$ be a $\Gamma$-proper cocompact $spin^c$ manifold of dimension $n$ and let $\omega\in \Omega_c^n(M\times \Gamma^p)$, for $p\in \mathbb{N}$. We consider
$\pi_M!(\omega)\in \mathbb{C}\Gamma^p$ given by
\begin{equation}
\pi_M!(\omega)(\gamma):=\int_M\omega(m,\gamma)
\end{equation}
for $\gamma\in \Gamma^p$. We have the following simple but illustrative lemma.
\begin{lemma}\label{lemmaintegrationfibers}
Let $M$ be a $\Gamma$-proper cocompact $spin^c$ manifold of dimension $n$. For every $p\in\mathbb{N}$, the map
\begin{equation}
\pi_M!:\Omega_c^n(M\times \Gamma^p)\longrightarrow \mathbb{C}\Gamma^p
\end{equation}
induce a well-defined morphism
\begin{equation}
\pi_M!:H_c^p(M\rtimes \Gamma)\longrightarrow H_p(\Gamma).
\end{equation}

\end{lemma}
\begin{proof}
For $G=M\rtimes \Gamma \rightrightarrows M$, the de Rham complex with compact supports (\ref{dRcomplexcompsupp}), where $dim\, M=n$, admits the following form
\begin{equation}\label{dRcomplexcompsuppMtimesGamma}
\xymatrix{
\cdots	\ar[r]^-{\partial_!}&\Omega_c^{n}(M\times \Gamma^2)\ar[r]^-{\partial_!}&\Omega_c^{n}(M\times\Gamma)\ar[r]^-{\partial_!}&\Omega_c^{n}(M)\\
\cdots	\ar[r]^-{\partial_!}&\Omega_c^{n-1}(M\times\Gamma^2)\ar[r]^-{\partial_!}\ar[u]^-d&\Omega_c^{n-1}(M\times\Gamma)\ar[u]^-{-d}\ar[r]^-{\partial_!}&\Omega_c^{n-1}(M)\ar[u]^-d\\
&\vdots\ar[u]^-d&\vdots\ar[u]^-{-d}&\vdots \ar[u]^-d
}
\end{equation}
because $(M\rtimes \Gamma)^{(p)}$ identifies canonically with $M\times \Gamma^p$ and $r=0$ in this case. Under this identification one can directly compute the compatibility between the top horizontal complex in (\ref{dRcomplexcompsuppMtimesGamma}) and the complex defining group homology
\begin{equation}
\xymatrix{
\cdots	\ar[r]^-{\partial_!}&\mathbb{C} \Gamma^2\ar[r]^-{\partial_!}&\mathbb{C}\Gamma\ar[r]^-{\partial_!}&\mathbb{C}\\
}
\end{equation}
via the maps $\pi_M$, that is, a direct computation verifies $\pi_M\circ \partial_!=\partial_!\circ \pi_M$ since the structural maps from
$M\times \Gamma^p$ to $M\times \Gamma^{p-1}$ defining $\partial_!$ are given precisely by the structural maps for the group $\Gamma$ in the second coordinate.
Now, for forms in $\Omega_c^q(M\times \Gamma^p)$ with $q< n$, the map $\pi_M$ is defined as zero. Hence, in order to conclude we just need to observe that $\pi_M \circ d=0$, which is a result of classic integration along the fibers in de Rham cohomology for manifolds.
\end{proof}
Applying Lemma~\ref{lemmaintegrationfibers} to every $M_g\rtimes \Gamma_g$ (for a fixed $g$ in its conjugacy class) for finite order elements $g$, we get a well-defined morphism, for $*=0,1$,
\begin{equation}
\xymatrix{
H^*_{\Gamma,deloc}(M)\ar[r]^-{\pi_M!}&H_*(\Gamma,F\Gamma),
}
\end{equation}
for every $M$ as above of dimension $n=*, mod\, 2$. We use the same notation $\pi_M!$ for the delocalized and localized integration along the fibers, because it should be clear from the context.

We are ready to state and prove the main theorem of this section.

\begin{theorem}\label{thmmuF}
The maps
\begin{equation}
\xymatrix{
H^*_{\Gamma,deloc}(M)\ar[r]^-{\pi_M!}&H_*(\Gamma,F\Gamma),
}
\end{equation}
induce a well-defined assembly map
\begin{equation}
\mu_{F\Gamma}: H^*_{top}(\Gamma)\longrightarrow H_*(\Gamma,F\Gamma)
\end{equation}
given by
\begin{equation}
\mu_{F\Gamma}([M,\omega])=\pi_M!(\omega).
\end{equation}
\end{theorem}

The proof of the theorem will require some preliminary lemmas and propositions. Before stating the first lemma, recall that for a $\Gamma$-manifold $M$ (proper, cocompact and $spin^c$) there is also a canonical pullback morphism
\begin{equation}
\xymatrix{
H^*(\Gamma)\ar[r]^-{\pi_M^*}&H^*(M\rtimes \Gamma),
}
\end{equation}
induced from the canonical groupoid morphism
$$M\rtimes \Gamma \to \Gamma.$$

\begin{lemma}\label{auxlemma0}
Let $M$ be a $\Gamma$-proper cocompact $spin^c$ manifold. We have the following equality of morphisms
\begin{equation}
(\pi_M!)^*=\pi_M^*.
\end{equation}
\end{lemma}
\begin{proof}
 Every functional $\tilde{h}:\mathbb{C}\Gamma^p\to \mathbb{C}$ is completely determined by an element $h\in \mathbb{C}\Gamma^p$, explicitly given by $\tilde{h}(\sum a_\gamma\gamma)=\sum a_\gamma h(\gamma)$.

For every $f\in \mathbb{C}\Gamma^p$ there is an associated current $C_{\pi_M^*f}: \Omega^n(M\times \Gamma^p)\to \mathbb{C}$ explicitly described in Example \ref{examplecapproduct} and from which we have
\begin{equation}
(\pi_M)^*(\tilde{f}) (\omega):=C_{\pi_M^*f}(\omega)=\int_M \omega \cap \pi^*f=\sum_\gamma \left(\int_M\omega(m,\gamma)\cdot f(\gamma)\right)=\tilde{f}(\pi_M!(\omega))=:(\pi_M!)^*(\tilde{f})(\omega).
\end{equation}
for every $\omega\in \Omega^n_c(M\times \Gamma^p)$. This concludes the proof.
\end{proof}

Now we will need some extra terminology and recall some basic facts on free proper actions. In fact, the general strategy will be to justify, using Matthey and L\"uck's results, that one can reduce the computations to free and proper cycles.

Let $M$ be a $\Gamma$-free proper cocompact $spin^c$ manifold. In this case $M/\Gamma$ has a $C^\infty$-structure compatible with its quotient structure. Moreover, $M\rtimes \Gamma$ and $M/\Gamma$ (seen as unit groupoids) are Morita equivalent Lie groupoids. Then there is an induced isomorphism
\begin{equation}
\label{eq:MoritaH}
\xymatrix{
H^*_c(M\rtimes \Gamma) \ar[r]^-M_-\cong& H_{dR}^*(M/\Gamma),
}
\end{equation}
because de Rham groupoid cohomology (compactly supported as well) is indeed Morita invariant; See for example Behrend's \cite{Beh} or \cite{BGNX}.

Now, we can consider the isomorphism
\begin{equation}
\xymatrix{
H^*_c(M\rtimes \Gamma) \ar[r]^-{M_{PD}}_-\cong& H^{dR}_*(M/\Gamma).
}
\end{equation}
given by the composition of $M$ in (\ref{eq:MoritaH}) followed by the Poincar\'e Duality isomorphism (for the compact smooth manifold $M/\Gamma$)
\begin{equation}
\xymatrix{
H^*_{dR}(M/\Gamma) \ar[r]^-{PD}_-\cong& H^{dR}_*(M/\Gamma).
}
\end{equation}

We can state the following lemma about free and proper actions, the proof of which is an immediate consequence of Lemma~\ref{auxlemma0}.

\begin{lemma}\label{auxlemma1}
Let $M$ be a $\Gamma$-free proper cocompact $spin^c$ manifold. The following diagram is commutative
\begin{equation}
\xymatrix{
H^*_c(M\rtimes \Gamma) \ar[d]_-{M_{PD}}^-\cong \ar[r]^-{\pi_M!}& H_*(\Gamma) \\
H^{dR}_*(M/\Gamma)\ar[ru]_-{(\pi_M)_*} &
},
\end{equation}
where
\begin{equation}
\xymatrix{
H_*(M/\Gamma)\ar[r]^-{(\pi_M)_*}& H_*(\Gamma)
}
\end{equation}
is the universal morphism induced from the continuous universal classifying map
\begin{equation}
M/\Gamma \stackrel{\pi_M}{\longrightarrow} B\Gamma.
\end{equation}
\end{lemma}

We will also need to compare our pushforward maps in delocalized cohomology with the usual pushforward maps in homology whenever this makes sense. The result is the following.

\begin{lemma}\label{auxlemma2}
Let $f:M\to M'$ be a $\Gamma$-equivariant $C^\infty$-map between two  $\Gamma$-free proper cocompact $spin^c$ manifolds. The following diagram commutes
\begin{equation}
\xymatrix{
H^*_c(M\rtimes \Gamma) \ar[d]_-{M_{PD}}^-\cong \ar[r]^-{f_!}& H^*_c(M'\rtimes \Gamma) \ar[d]^-{M_{PD}}_-\cong \\
H^{dR}_*(M/\Gamma)\ar[r]_-{\overline{f}_*} & H^{dR}_*(M'/\Gamma),
}
\end{equation}
where $\overline{f}:M/\Gamma\to M'/\Gamma$ is the smooth map associated to $f$.
\end{lemma}

\begin{proof}
The first thing to recall is the commutative diagram involving Poincar\'e Duality 
\begin{equation}
\xymatrix{
H_{dR}^*(M/\Gamma)\ar[d]_-{PD}^-\cong\ar[r]^-{\overline{f}_!} & H_{dR}^*(M'/\Gamma)\ar[d]^-{PD}_-\cong\\
H^{dR}_*(M/\Gamma)\ar[r]_-{\overline{f}_*} & H^{dR}_*(M'/\Gamma),
}
\end{equation}
classically associated to smooth maps between smooth compact manifolds.
Then we need to justify the commutativity of the following diagram
\begin{equation}\label{Moritaf}
\xymatrix{
H_c^*(M\rtimes\Gamma)\ar[d]_-M^-\cong\ar[r]^-{f_!} & H_c^*(M'\rtimes\Gamma)\ar[d]^-{M}_-\cong \\
H_{dR}^*(M/\Gamma)\ar[r]_-{\overline{f}_!} & H_{dR}^*(M'/\Gamma).
}
\end{equation}
For this let us consider the associated diagram in $K$-theory
\begin{equation}
\label{eq:Morita!}
\xymatrix{
K^*(M\rtimes\Gamma)\ar[d]_-M^-\cong\ar[r]^-{f_!} & K^*(M'\rtimes\Gamma)\ar[d]^-{M}_-\cong \\
K^*(M/\Gamma)\ar[r]_-{\overline{f}_!} & K^*(M'/\Gamma).
}
\end{equation}
By standard $K$-theory techniques this last diagram is commutative. Indeed, the Morita bi-modules induce isomorphisms in $K$-theory that are compatible with the Thom isomorphisms and are natural. Hence, the isomorphisms are compatible with the shriek morphisms constructed using deformation groupoids.
Now, the commutativity of the diagram (\ref{Moritaf}) can be verified by diagram chasing using the K-theoretical diagram (\ref{eq:Morita!}) tensored with $\mathbb{C}$ together with the twisted Chern isomorphisms and the commutativity given by the corresponding Riemann-Roch theorems.
\end{proof}

For the next preliminary result we need to introduce a new $K$-theory group. For $\Gamma$, a countable discrete group we have already investigated the groups
$K^*_{top}(\Gamma)$. Now we will also consider the groups
\begin{equation}
K^*_{topf}(\Gamma)
\end{equation}
defined in an analog way as $K^*_{top}(\Gamma)$ but only used free proper manifolds instead  of proper ones. There is a canonical morphism
\begin{equation}
\varphi: K_{topf}^*(\Gamma)\longrightarrow K^*_{top}(\Gamma)
\end{equation}
given by forgetting that the action is free. We have the following proposition whose proof follows the same lines as Corollary~\ref{ConnesvsLuck}, but by considering instead free and proper actions.

\begin{proposition}\label{ConnesvsLuckfree}
Let $\Gamma$ be a countable discrete group. There is a unique isomorphism
\begin{equation}
\xymatrix{
K^*_{topf}(\Gamma)\ar[r]^-\lambda_-\cong & K_*^\Gamma(E\Gamma)\cong K_*(B\Gamma)
}
\end{equation}
fitting in the following commutative diagram of group isomorphisms
\begin{equation}
\xymatrix{
K_*(B\Gamma)\ar[r]^-\varphi & K_*^\Gamma(\underline{E\Gamma}) \\
K^*_{topf}(\Gamma)\ar[u]^-\lambda_-\cong \ar[r]_-\varphi & K_{top}^*(\Gamma)\ar[u]_-\lambda^-\cong
}
\end{equation}
where the morphism $\lambda$ on the right is the one in Corollary \ref{ConnesvsLuck} and where the morphism $\varphi$ on the top is the natural morphism associated to classifying spaces.
\end{proposition}

Now, remember we have cited above a concrete isomorphism (see (\ref{Mattheysiso}))
\begin{equation}
\xymatrix{
\bigoplus_{g\in \langle\Gamma\rangle^{fin}} K_*(B\Gamma_g)\otimes \mathbb{C}\ar[rr]^-{\phi_*^\Gamma}_-\cong &&K_*^\Gamma(\underline{E\Gamma})\otimes \mathbb{C}
}
\end{equation}
constructed by Matthey. See Section 7 and Theorem 1.3 in \cite{Matthey}.

We let
\begin{equation}\label{morphismphi}
\xymatrix{
\bigoplus_{g\in \langle\Gamma\rangle^{fin}} K_{topf}^*(\Gamma_g)\otimes \mathbb{C}\ar[rr]^-{\phi_*^\Gamma}_-\cong &&K_{top}^*(\Gamma)\otimes \mathbb{C}
}
\end{equation}
be the unique isomorphism fitting the following commutative diagram
\begin{equation}
\xymatrix{
\bigoplus_{g\in \langle\Gamma\rangle^{fin}} K_*(B\Gamma_g)\otimes \mathbb{C}\ar[rr]^-{\phi_*^\Gamma}_-\cong &&K_*^\Gamma(\underline{E\Gamma})\otimes \mathbb{C} \\
\bigoplus_{g\in \langle\Gamma\rangle^{fin}} K_{topf}^*(\Gamma_g)\otimes \mathbb{C}\ar[u]^-\lambda_-\cong\ar[rr]_-{\phi_*^\Gamma}^-\cong &&K_{top}^*(\Gamma)\otimes \mathbb{C}. \ar[u]_-\lambda^-\cong
}
\end{equation}

Now, the strategy of proving Theorem \ref{thmmuF} can be used to justify that
\begin{equation}\label{strategyML}
\xymatrix{
\bigoplus_{g\in \langle\Gamma\rangle^{fin}} K_{topf}^*(\Gamma_g)\otimes \mathbb{C} \ar[rr]^-{ch_{top}\circ \phi_*^\Gamma}&&H_{top}^*(\Gamma)\ar[r]^-{\mu_{F\Gamma}}& H_*(\Gamma,F\Gamma)
}
\end{equation}
is well-defined, which will imply, since $ch_{top}\circ \phi_*^\Gamma$ is a well-defined isomorphism, that $\mu_{F\Gamma}$ is well-defined. For this we will then need an explicit description of $\phi_*^\Gamma$ which is possible thanks to the work of Matthey. Indeed, in Section~7 of \cite{Matthey} Matthey defined (Definition~7.2 ref.cit.)
\begin{equation}
\xymatrix{
\bigoplus_{g\in \langle\Gamma\rangle^{fin}} K_*(B\Gamma_g)\otimes \mathbb{C}\ar[rr]^-{\phi_*^\Gamma}_-\cong &&K_*^\Gamma(\underline{E\Gamma})\otimes \mathbb{C}
}
\end{equation}
explicitly as the composition
{\tiny
\begin{equation}
\label{eq:psi.chi.j}
\xymatrix{
\bigoplus_{g\in \langle\Gamma\rangle^{fin}} K_*(B\Gamma_g)\otimes \mathbb{C} \ar[r]^-\varphi &\bigoplus_{g\in \langle\Gamma\rangle^{fin}} K^{\Gamma_g}_*(\underline{E\Gamma_g})\otimes \mathbb{C}\ar[r]^-\chi_-\cong & \bigoplus_{g\in \langle\Gamma\rangle^{fin}} K^{\Gamma_g}_*(\underline{E\Gamma_g})\otimes \mathbb{C}
\ar[r]^-j & K_*^\Gamma(\underline{E\Gamma})\otimes \mathbb{C}
}
\end{equation}
}
where $\varphi$ is the canonical map associated to the canonical map $E\Gamma_g\to \underline{E\Gamma_g}$, $j$ is the morphism associated to the group inclusions $\Gamma_g\to \Gamma$, and 
\begin{equation}\label{isochiMatthey}
\xymatrix{
\bigoplus_{g\in \langle\Gamma\rangle^{fin}} K^{\Gamma_g}_*(\underline{E\Gamma_g})\otimes \mathbb{C}\ar[r]^-\chi_-\cong & \bigoplus_{g\in \langle\Gamma\rangle^{fin}} K^{\Gamma_g}_*(\underline{E\Gamma_g})\otimes \mathbb{C}
}
\end{equation}
is an isomorphism that we now describe. Matthey started by showing that if $g$ acts trivially on a Hausdorff proper $\Gamma_g$-space $X$, then there is a canonical isomorphism (Lemma~8.1 in \cite{Matthey})
\begin{equation}\label{decompostionMattheylemma8.1}
K^{\Gamma_g}_*(X)\cong \bigoplus_{l=0,...,n_g-1} K^{\Gamma_g}_*(X)_l
\end{equation}
where
\begin{equation}
K^{\Gamma_g}_*(X)_l:= \{[H,\pi, F]: g \, \,\text{acts on}\,\,  H \,\,\text{by}\,\, \omega_g^l \}
\end{equation}
with $\omega_g=e^{2\pi i/n_g}$ and with the usual notations for Kasparov cycles. As detailed in Appendix~B of \cite{Matthey}, there is at least one model for $\underline{E\Gamma_g}$ on which $g$ acts trivially. In Proposition 8.2 in ref.cit. he computed explicitly the morphism $\chi$ above with respect to the decomposition (\ref{decompostionMattheylemma8.1}), he showed that
\begin{equation}
\chi_{g,l}(x\otimes \lambda)=x\otimes \omega_g^l\cdot \lambda,
\end{equation}
where, with respect to the decomposition above,
\begin{equation}\label{decompositionchiMatthey}
\chi=\bigoplus_{g\in \langle\Gamma\rangle^{fin},l=0,...,n_g} (\chi)_{g,l}.
\end{equation}
In particular he obtained that $\chi$ is an isomorphism with $\chi$ itself as an inverse.

Coming back to our situation, we use Proposition~\ref{ConnesvsLuck} to get a unique isomorphism
\begin{equation}\label{chitopiso}
\xymatrix{
\bigoplus_{g\in \langle\Gamma\rangle^{fin}} K^*_{top}(\Gamma_g)\otimes \mathbb{C}\ar[r]^-\chi_-\cong & \bigoplus_{g\in \langle\Gamma\rangle^{fin}} K^*_{top}(\Gamma_g)\otimes \mathbb{C}
}
\end{equation}
fitting in the commutative diagram
\begin{equation}\label{decompositionchiMattheyTOP}
\xymatrix{
\bigoplus_{g\in \langle\Gamma\rangle^{fin}} K^{\Gamma_g}_*(\underline{E\Gamma_g})\otimes \mathbb{C}\ar[r]^-\chi_-\cong & \bigoplus_{g\in \langle\Gamma\rangle^{fin}} K^{\Gamma_g}_*(\underline{E\Gamma_g})\otimes \mathbb{C}\\
\bigoplus_{g\in \langle\Gamma\rangle^{fin}} K^*_{top}(\Gamma_g)\otimes \mathbb{C}\ar[u]^-\lambda_-\cong\ar[r]_-\chi^-\cong & \bigoplus_{g\in \langle\Gamma\rangle^{fin}} K^*_{top}(\Gamma_g)\otimes \mathbb{C}, \ar[u]_-\lambda^-\cong
}
\end{equation}
but more importantly admitting a decomposition (if we take $\underline{E\Gamma_g}$ on which $g$ acts trivially as above)
{\tiny
\begin{equation}\label{chigl}
\xymatrix{
\chi=\bigoplus_{g\in \langle\Gamma\rangle^{fin},l=0,...,n_g} (\chi)_{g,l}\ar@{}[r]|:&\bigoplus_{g\in \langle\Gamma\rangle^{fin},l=0,...,n_g}K^*_{top}(\Gamma_g)_l\otimes \mathbb{C}\ar[r]& \bigoplus_{g\in \langle\Gamma\rangle^{fin},l=0,...,n_g}K^*_{top}(\Gamma_g)_l\otimes \mathbb{C}.
}
\end{equation}
}
compatible with the decomposition in factors (\ref{decompositionchiMatthey}) and computed explicitly by
\begin{equation}\label{chiglexplicit}
(\chi)_{g,l}([M,x]\otimes \lambda)=[M,x]\otimes \omega_g^l\cdot \lambda
\end{equation}
for $[M,x]\in K^*_{top}(\Gamma_g)_l\subset K^*_{top}(\Gamma_g)$, where the factor $K^*_{top}(\Gamma_g)_l$ corresponds to the factor $K^{\Gamma_g}_*(\underline{E\Gamma_g})_l$ via the isomorphism $\lambda$ through Matthey's decomposition (\ref{decompostionMattheylemma8.1}).

We are ready to prove the main theorem of this section.

\begin{proof}[\bf of Theorem \ref{thmmuF}]
We need to justify that
\begin{equation}\label{strategyML}
\xymatrix{
\bigoplus_{g\in \langle\Gamma\rangle^{fin}} K_{topf}^*(\Gamma_g)\otimes \mathbb{C} \ar[rr]^-{ch_{top}\circ \phi_*^\Gamma}&&H_{top}^*(\Gamma)\ar[r]^-{\mu_{F\Gamma}}& H_*(\Gamma,F\Gamma)
}
\end{equation}
is well-defined, which will imply, since $ch_{top}\circ \phi_*^\Gamma$ is a well-defined isomorphism, that $\mu_{F\Gamma}$ is well-defined.

By the previous discussion, the isomorphism $\phi_*^\Gamma$ appearing in (\ref{strategyML}) fits in the following commutative diagram
{\tiny
\begin{equation}
\xymatrix{
&&&\\
\bigoplus_{g\in \langle\Gamma\rangle^{fin}} K_*(B\Gamma_g)\otimes \mathbb{C} \ar@/^3pc/[rrr]^-{\phi_*^\Gamma}_-\cong\ar[r]^-\varphi &\bigoplus_{g\in \langle\Gamma\rangle^{fin}} K^{\Gamma_g}_*(\underline{E\Gamma_g})\otimes \mathbb{C}\ar[r]^-\chi_-\cong & \bigoplus_{g\in \langle\Gamma\rangle^{fin}} K^{\Gamma_g}_*(\underline{E\Gamma_g})\otimes \mathbb{C}
\ar[r]^-j & K_*^\Gamma(\underline{E\Gamma})\otimes \mathbb{C}\\
\bigoplus_{g\in \langle\Gamma\rangle^{fin}} K^*_{topf}(\Gamma_g)\otimes \mathbb{C} \ar@/_3pc/[rrr]_-{\phi_*^\Gamma}^-\cong \ar[u]^-\lambda_-\cong\ar[r]^-\varphi &\bigoplus_{g\in \langle\Gamma\rangle^{fin}} K^*_{top}(\Gamma_g)\otimes \mathbb{C}\ar[u]^-\lambda_-\cong\ar[r]^-\chi_-\cong & \bigoplus_{g\in \langle\Gamma\rangle^{fin}} K^*_{top}(\Gamma_g)\otimes \mathbb{C}\ar[u]^-\lambda_-\cong
\ar[r]^-j & K^*_{top}(\Gamma)\otimes \mathbb{C}\ar[u]^-\lambda_-\cong
&&&\\
}
\end{equation}
}
where the morphisms $\varphi$ and $\chi$ are the same as in (\ref{eq:psi.chi.j}) and $j$ is the morphism induced by the group inclusions $\Gamma_g\subset \Gamma$.
We need then to show that $\mu_{F\Gamma}\circ ch_{top}\circ (j\circ \chi\circ \varphi)$ is well-defined. We will achieve this goal by decomposing the problem in three steps.

{\bf Step I: $\left( \bigoplus_{g\in \langle\Gamma\rangle^{fin}} \mu_{F\Gamma_g}\right) \circ \left( \bigoplus_{g\in \langle\Gamma\rangle^{fin}} ch_{\Gamma_g}\right)\circ \varphi$ is well-defined.} To show this, let us fix an element $g\in \Gamma$ of finite order, let $M$ be a $\Gamma_g$-free proper cocompact $spin^c$ manifold, and let $x\in K_{\Gamma_g}^*(M)$. We want to prove that
\begin{equation}
(M,x)\mapsto \pi_{M}!(ch_g(x)\wedge Td_g)=\mu_{F\Gamma_g}\circ ch_{\Gamma_g}\circ \varphi_g
\end{equation}
gives rise to a well-defined morphism
\begin{equation}\label{auxmorphism}
K_{topf}^*(\Gamma_g)\longrightarrow H_*(\Gamma_g),
\end{equation}
where $\varphi_g:K^*_{topf}(\Gamma_g)\otimes \mathbb{C} \to K^*_{top}(\Gamma_g)\otimes \mathbb{C}$ is the $g$-component of $\varphi$, {\it i.e.}, $\varphi=\bigoplus_{g\in \langle\Gamma\rangle^{fin}} \varphi_g$.
Hence, let us consider a $\Gamma_g$-equivariant map $f:M\to M'$ between two $\Gamma_g$-free proper cocompact $spin^c$ manifolds. We have the following diagram
\begin{equation}
\label{eq:commdiag}
{\small
\xymatrix{
K^*(M\rtimes \Gamma_g)\ar[dd]_-{f_!}\ar[r]^-{ch_g\wedge Td_g}&
H_c^*(M\rtimes \Gamma_g)\ar[dd]^-{f_!}\ar[rd]^-{\pi_{M}!}&\\
&&H_*(\Gamma_g)\\
K^*(M'\rtimes \Gamma_g)\ar[r]_-{ch_g\wedge Td_g}& H_c^*(M'\rtimes \Gamma_g)
\ar[ru]_-{\pi_{M'}!}&
}}
\end{equation}
whose commutativity precisely implies that (\ref{auxmorphism}) is well-defined. Now, the first square in the diagram (\ref{eq:commdiag}) commutes by the Riemann-Roch Theorem~\ref{RRthm} and the triangle on the right commutes by Lemmas~\ref{auxlemma1} and~\ref{auxlemma2}.

{\bf Step 2: $\left( \bigoplus_{g\in \langle\Gamma\rangle^{fin}} \mu_{F\Gamma_g}\right) \circ \left( \bigoplus_{g\in \langle\Gamma\rangle^{fin}} ch_{\Gamma_g}\right)\circ \chi \circ \varphi$ is well defined.} To see this we shall use the decomposition of $\chi$ as in (\ref{decompositionchiMattheyTOP}) and its computation in (\ref{chigl}). By fixing an element $\gamma \in \Gamma$ of finite order and $l\in \{0,1,...,n_\gamma\}$ we need to show that
\begin{equation}\label{step2morphism}
\mu_{F\Gamma_\gamma,l} \circ ch^{\Gamma_\gamma}_{\gamma,l}\circ \chi_{\gamma,l}\circ \varphi_{\gamma,l}
\end{equation}
is a well-defined morphism
\begin{equation}
\label{eq:chernFGamma}
\bigoplus_{g\in \langle\Gamma\rangle^{fin}} K^*_{topf}(\Gamma_g)\otimes \mathbb{C}\longrightarrow H_*(\Gamma_\gamma,F\Gamma_\gamma),
\end{equation}
where
$$\varphi_{\gamma,l}:\bigoplus_{g\in \langle\Gamma\rangle^{fin}} K^*_{topf}(\Gamma_g)\otimes \mathbb{C}\to K^*_{top}(\Gamma_\gamma)_l\otimes \mathbb{C}$$
is given by $\varphi_{\gamma,l}(x):=(\varphi(x))_{\gamma,l}$ with respect to the decomposition (\ref{decompositionchiMattheyTOP}), and
$$ch^{\Gamma_\gamma}_{\gamma,l}:K^*_{top}(\Gamma_\gamma)_l\otimes \mathbb{C} \longrightarrow H^*_{top}(\Gamma_\gamma)_l$$
is just the $l$-component in
$$ch^{\Gamma_\gamma}=\bigoplus_{l=0,...,n_\gamma}ch^{\Gamma_\gamma}_{\gamma,l}$$
of the isomorphism $ch^{\Gamma_\gamma}:K_{top}(\Gamma_\gamma)\otimes \mathbb{C}\longrightarrow H_{top}^*(\Gamma_\gamma)$ with respect to the isomorphism (\ref{chitopiso}) and its decomposition (\ref{decompositionchiMattheyTOP}), and where
$$\mu_{F\Gamma_\gamma,l}:H^*_{top}(\Gamma_\gamma)_l\to H_{top}(\Gamma_\gamma,F\Gamma_\gamma)$$
denotes the restriction of $\mu_{F\Gamma_\gamma}$ to the $l$-component. Let
$$\chi_{\gamma,l}:H_*(\Gamma_\gamma,F\Gamma_\gamma)\to H_*(\Gamma_\gamma,F\Gamma_\gamma)$$
be the morphism given by the multiplication by $\omega_g^l$. By step 1, the morphism given by the composition as in the following diagram

\begin{equation}\label{step2aux}
\xymatrix{
\bigoplus_{g\in \langle\Gamma\rangle^{fin}} K^*_{topf}(\Gamma_g)\otimes \mathbb{C}\ar[r]^-{\varphi_{\gamma,l}} &
K^*_{top}(\Gamma_\gamma)_l \otimes \mathbb{C}\ar[d]^-{ch^{\Gamma_\gamma}_{\gamma,l}}_-\cong &\\
&H^*_{top}(\Gamma_\gamma)_l\ar[d]^-{\mu_{F\gamma,l}}& \\
&H_{top}(\Gamma_\gamma,F\Gamma_\gamma)\ar[r]^-{\chi_{\gamma,l}}_-\cong&H_{top}(\Gamma_\gamma,F\Gamma_\gamma)\\
}
\end{equation}
is well-defined.

By considering the morphism $\chi_{\gamma,l}$ in (\ref{chiglexplicit}) appearing in the decomposition (\ref{chigl}) we obtain that the following diagram
{\small
\begin{equation}
\xymatrix{
\bigoplus_{g\in \langle\Gamma\rangle^{fin}} K^*_{topf}(\Gamma_g)\otimes \mathbb{C}\ar[r]^-{\varphi_{\gamma,l}}&
K^*_{top}(\Gamma_\gamma)_l\otimes \mathbb{C}\ar[d]_-{ch^{\Gamma_\gamma}_{\gamma,l}}^-\cong\ar[r]^-{\chi_{\gamma,l}}&K^*_{top}(\Gamma_\gamma)_l\otimes \mathbb{C}\ar[d]^-{ch^{\Gamma_\gamma}_{\gamma,l}}_-\cong\\
&H^*_{top}(\Gamma_\gamma)_l\ar[d]_-{\mu_{F\gamma,l}}\ar[r]_-{\chi_{\gamma,l}}^-\cong&H^*_{top}(\Gamma_\gamma)_l\ar[d]^-{\mu_{F\gamma,l}}\\
&H_{*}(\Gamma_\gamma,F\Gamma_\gamma)\ar[r]_-{\chi_{\gamma,l}}^-\cong&H_{*}(\Gamma_\gamma,F\Gamma_\gamma),
}
\end{equation}}
where the (co)homological $\chi_{\gamma,l}$ are given by multiplication by $\omega_g^l$, is commutative. Then, from (\ref{step2aux}) above, we have that (\ref{step2morphism}) is well-defined.

{\bf Step 3: (\ref{strategyML}) is well-defined.}

By naturality of the Chern character morphisms, we have the following commutative diagram
\begin{equation}\label{step3diagramj}
\xymatrix{
\bigoplus_{g\in \langle\Gamma\rangle^{fin}}K^*_{top}(\Gamma_g)\otimes \mathbb{C}\ar[r]^-j\ar[dd]_-{\bigoplus_{g\in \langle\Gamma\rangle^{fin}} ch^{\Gamma_g}}^-\cong& K^*_{top}(\Gamma)\ar[dd]^-{ch^\Gamma}_-\cong\otimes \mathbb{C} \\
&\\
\bigoplus_{g\in \langle\Gamma\rangle^{fin}} H^*_{top}(\Gamma_g)\ar[r]_-j&H^*_{top}(\Gamma).
}
\end{equation}
For showing that $\mu_{F\Gamma}\circ ch^\Gamma\circ (j\circ \chi\circ \varphi)$ in (\ref{strategyML}) is well-defined, it is then enough to show that
\begin{equation}\label{strategyML2}
\xymatrix{
\bigoplus_{g\in \langle\Gamma\rangle^{fin}} K_{topf}^*(\Gamma_g)\otimes \mathbb{C} \ar[rrr]^-{ \left( \bigoplus_{g\in \langle\Gamma\rangle^{fin}} ch_{\Gamma_g}\right)\circ (\chi \circ \varphi)}&&&
\bigoplus_{g\in \langle\Gamma\rangle^{fin}} H_{top}^*(\Gamma_g)\ar[dd]^-{\left( \bigoplus_{g\in \langle\Gamma\rangle^{fin}} \mu_{F\Gamma_g}\right)}\\
&&&\\
&&& \bigoplus_{g\in \langle\Gamma\rangle^{fin}}H_*(\Gamma_g,F\Gamma_g)
}
\end{equation}
is well-defined, which was proved in Step 2, because then by considering the canonical morphism
\begin{equation}
\xymatrix{
\bigoplus_{g\in \langle\Gamma\rangle^{fin}}H_*(\Gamma_g,F\Gamma_g) \ar[r]^-j& H_*(\Gamma,F\Gamma)
}
\end{equation}
induced by the inclusions $\Gamma_g\subset \Gamma$ we see that the morphism given by the composition  of (\ref{strategyML2}) followed by the morphism $j$ above is well-defined but this composition coincides, in cycles, precisely with $\mu_{F\Gamma}\circ ch^\Gamma\circ (j\circ \chi\circ \varphi)$ by the commutativity of the diagram~(\ref{step3diagramj}).

This concludes the proof.
\end{proof}

\subsubsection{The pairing formula}

We are ready to give a precise formula for the pairing between the topological $K$-theory of $\Gamma$ and the cohomology $H^*(\Gamma,F\Gamma)$, which will be recalled in the corollary below. Let us first notice that given a $\Gamma$-proper cocompact $spin^c$ manifold $M$ we have that for every $\omega\in H^*_c(M\rtimes\Gamma)$ and for every $c\in H^*(\Gamma)$, the following equality
\begin{equation}
\langle \pi_M!(\omega),c\rangle=\langle \omega,\pi_M^*c\rangle
\end{equation}
holds by Lemma~\ref{auxlemma0}, where the left hand side is the pairing between group cycles and group cocycles, and the right hand side is the pairing between invariant forms and currents.

The following corollary summarizes all the main results in the present article.

\begin{corollary}\label{Maincorollarygrouphomoogy}
For any discrete group $\Gamma$, there is a well-defined morphism
\begin{equation}
\xymatrix{
K^*_{top}(\Gamma)\ar[rr]^-{ch_{\mu_{F\Gamma}}^{\Gamma}}&&H_*(\Gamma,F\Gamma)\\
}
\end{equation}
given by the composition of the Chern character for $\Gamma$
\begin{equation}
ch^{\Gamma}: K^*_{top}(\Gamma) \longrightarrow H^*_{top}(\Gamma)
\end{equation}
followed by the cohomological assembly
\begin{equation}
\mu_{F\Gamma}: H^*_{top}(\Gamma)\longrightarrow H_*(\Gamma,F\Gamma).
\end{equation}
The morphism induces a  well-defined pairing
\begin{equation}
K^*_{top}(\Gamma)\times H^*(\Gamma,F\Gamma)\to \mathbb{C}
\end{equation}
given by
\begin{equation}
\langle[M,x],\tau\rangle:=\langle\mu_{F\Gamma}(ch_{M}^{\Gamma}(x)\wedge Td_M^\Gamma),\tau\rangle=
\langle \pi_M!(ch_{M}^{\Gamma}(x)\wedge Td_M^\Gamma),\tau\rangle,
\end{equation}
and computed in every $g$-component by
\begin{equation}\label{conjecturedformula}
\langle [M,x],\tau_g\rangle = \langle ch_M^g(x)\wedge Td_g^M,\pi_g^*(\tau_g)\rangle,
\end{equation}
where the pairing on the right hand side corresponds the pairing between $\Gamma_g$-invariant forms and currents on $M_g$.
\end{corollary}

\subsection{The Chern-Baum-Connes assembly map in periodic cyclic homology and the Chern-Connes pairing formula for discrete groups}

As mentioned above the main result of this paper is already given by Corollary~\ref{Maincorollarygrouphomoogy}, which gives a well-defined pairing of the left hand side of the Baum-Connes assembly map and of the delocalized homology group of a discrete group. Moreover, it gives a geometric formula for all possible cycles. Now, it is however interesting, as we have briefly discussed above, that the above pairing can also be stated as a pairing in cyclic periodic cohomology. In fact, for a countable discrete group $\Gamma$, Burghelea showed in Theorem $I_p'$ in \cite{Bur} that there is an isomorphism
\begin{equation}
HP_*(\mathbb{C}(\Gamma))\cong H_*(\Gamma,F\Gamma)\bigoplus T_*(\Gamma),
\end{equation}
where $T_*(\Gamma)$ is a group depending only on the conjugacy classes of infinite order elements of $\Gamma$ and will not be involved in this paper. In particular, there is a morphism
\begin{equation}
\xymatrix{
H_*(\Gamma,F\Gamma)\ar[r]^-B & HP_*(\mathbb{C}\Gamma)
}
\end{equation}
which is an isomorphism onto its image as a direct factor and we can hence consider the assembly map
\begin{equation}
\mu_{\Gamma}: H^*_{top}(\Gamma)\longrightarrow HP_*(\mathbb{C}\Gamma)
\end{equation}
given by the composition of the cohomological assembly map $\mu_{F\Gamma}$ of Theorem~\ref{thmmuF} followed by Burghelea's morphism $B$.

Recall that the periodic cyclic homology groups of $\mathbb{C}\Gamma$ have also been computed by Burghelea. He showed in \cite{Bur} that there are isomorphisms
\begin{equation}\label{groupalgebracohomologysection}
HP^*(\mathbb{C}\Gamma)\cong \left( \prod_{\langle \Gamma\rangle^{fin}}\bigoplus_{k=*,\,mod\, 2}H^k(\Gamma_g;\mathbb{C})\right)  \bigoplus T^*(\Gamma)
\end{equation}
 where again the group $T^*(\Gamma)$ depends only on the conjugacy classes of infinite order elements of $\Gamma$ and will not be needed in this paper. In any case,  Theorem~\ref{Maincorollarygrouphomoogy} can be restated in terms of periodic cyclic theory as follows.

\begin{theorem}[The Chern-Baum-Connes assembly map in periodic cyclic homology]\label{ChernassemblythmHP}
For any discrete group $\Gamma$, there is well defined morphism
\begin{equation}
\xymatrix{
K^*_{top}(\Gamma)\ar[rr]^-{ch_\mu^{\Gamma}}&&HP_*(\mathbb{C}\Gamma)\\
}
\end{equation}
given by the composition of the Chern character
\begin{equation}
ch^{\Gamma}: K^*_{top}(\Gamma) \longrightarrow H^*_{top}(\Gamma)
\end{equation}
followed by the cohomological assembly
\begin{equation}
\mu_{\Gamma}: H^*_{top}(\Gamma)\longrightarrow HP_*(\mathbb{C}\Gamma).
\end{equation}
The morphism induces a well-defined pairing
\begin{equation}
K^*_{top}(\Gamma)\times HP^*(\mathbb{C}\Gamma)\to \mathbb{C}
\end{equation}
computed in every $g$-component (with respect to Burghelea's decomposition above) by
\begin{equation}\label{pairingformulaintro}
\langle [M,x],\tau_g\rangle = \langle ch_M^g(x)\wedge Td_g^M,\pi_g^*(\tau_g)\rangle,
\end{equation}
where the pairing on the right hand side corresponds the pairing between $\Gamma_g$-invariant forms and currents on $M_g$.
We name the morphism $ch_\mu^{\Gamma}$ as the Chern-Baum-Connes assembly map of the group $\Gamma$.
\end{theorem}

There are several reasons for which we decided to state the final result Theorem~\ref{ChernassemblythmHP} in terms of a pairing with cyclic periodic cohomology instead of just finishing the paper with the delocalized group homology version:
\begin{enumerate}
\item An important part of the techniques developed in this paper apply to more general groups or groupoids. Computations such as the one by Burghelea for discrete groups are not so common, but several examples of geometric cyclic cocycles are known in many different situations. One could expect to extend some index theoretic formulas to some other settings by using the techniques of the present paper.
\item For geometric/topological applications, it would be interesting to explore the injectivity of the cohomological assembly maps $\mu_{F\Gamma}$ and $\mu_\Gamma$. In principle, there are no analytic obstructions for this kind of problems. In fact, there are at least two models of similar cohomological assembly maps of this kind, both using periodic cyclic theory, in which the injectivity was proved under general assumptions on the group; See \cite{CorTar1} and \cite{Yualgebraic}. A comparison of their work with our cohomological assembly maps would be very interesting.
\item  Given a periodic cyclic class $\tau\in HP^*(\mathbb{C}\Gamma)$, by pairing with $K$-theory, it gives rise to a morphism
$\varphi_\tau: K_*(\mathbb{C}\Gamma)\to \mathbb{C}$; 	
Sometimes it makes sense to replace the domain by the $K$-theory groups of larger algebras (for example, the algebra $\mathcal{R}\Gamma$ considered by Connes and Moscovici in \cite{CMnov}). If the morphism $\varphi_\tau$ extends to $K_*(C_r^*(\Gamma))$ (or at least to appropriate smooth subalgebras) then one can consider the pairing
\begin{equation}
\langle\mu_{BC}(x),\tau\rangle
\end{equation}
with $x\in K_{top}^*(\Gamma)$ and $\mu_{BC}$ the $K$-theoretical Baum-Connes assembly map. Note that extending the morphisms $\varphi_\tau$ is a very hard problem. For example, for classes localized at identity, this would imply the Novikov conjecture on the homotopy invariance of higher signatures. On the other hand, for $x\in K_{top}^*(\Gamma)$ and any $\tau$ as above our pairing, $\langle ch_{\mu}^\Gamma(x),\tau\rangle$, is defined (and computed explicitly by Theorem~\ref{ChernassemblythmHP}) and in all known cases we have
\begin{equation}
\label{eq:pairing.BC}
\langle ch_{\mu}^\Gamma(x),\tau\rangle=\langle\mu_{BC}(x),\tau\rangle,
\end{equation}
whenever the right hand side makes sense and whenever it has been computed. This rises a very interesting question, in general, if $\varphi_\tau$ extends, do we have the equality (\ref{eq:pairing.BC})? 
\item In recent years, very interesting articles have appeared on the study of higher secondary invariants, surgery sequences and cyclic cocycle pairings and the very interesting and fruitful relations between these three topics were developed; For example, just to mention some of them: \cite{Ben}, \cite{BenRoy1}, \cite{CWXY}, \cite{PSZ} and \cite{XieYualgebraicity}. All these works are certainly closely related to our present paper and it would be highly interesting to explore some relations with these results.
\end{enumerate}



\appendix

\section{Semidirect product groupoids and algebras used in this article}\label{appendixgrpdalgebras}

In this section we will describe explicitly the main groupoids we use in the present article as well as the associated algebras.

Let $\Gamma$ be a discrete group and $M$ be a $\Gamma$-manifold, we can consider the following groupoids and algebras:

\begin{itemize}
\item {\bf The semi-direct groupoid $M\rtimes \Gamma$:} It is the groupoid
\begin{equation}
M\rtimes \Gamma\rightrightarrows M
\end{equation}
with structure maps
\begin{enumerate}
\item Source and target maps defined by
$$s(m,g)=m\cdot g,\,\,\, t(m,g)=m,$$
\item groupoid product
$$(m,g)\cdot (mg,h):=(m,gh),$$
\item unit map
$$u(m)=(m,e),$$
where $e\in \Gamma$ is the neutral element for the group structure, and
\item inverse map
$$(m,g)^{-1}:=(mg,g^{-1}).$$
\end{enumerate}
For this groupoid we will usually consider the usual groupoid convolution algebra of complex valued compactly supported $C^\infty$-functions $C_c^\infty(M\rtimes \Gamma)$.
\end{itemize}

\begin{itemize}
\item {\bf The semi-direct groupoid $(M\times M)\rtimes \Gamma$:} It is the groupoid associated to the diagonal action on $M\times M,$ given by
\begin{equation}
(M\times M)\rtimes \Gamma\rightrightarrows M
\end{equation}
with structure maps
\begin{enumerate}
\item Source and target maps defined by
$$s(m,n,g)=n\cdot g^{-1},\,\,\, t(m,,n,g)=m,$$
\item groupoid product
$$(m,n,g)\cdot (ng^{-1},l,h):=(m,lg,hg)$$
\item unit map
$$u(m)=(m,m,e),$$
where $e\in \Gamma$ is the neutral element for the group structure, and
\item inverse map
$$(m,n,g)^{-1}:=(ng^{-1},mg^{-1},g^{-1}).$$
\end{enumerate}
For this groupoid we will usually consider the usual groupoid convolution algebra of complex valued compactly supported $C^\infty$-functions $C_c^\infty((M\times M)\rtimes \Gamma)$.
\end{itemize}

\begin{itemize}
\item {\bf The semi-direct groupoid $TM \rtimes \Gamma$:} It is the groupoid associated to the infinitesimal action on $TM$, it is given by
\begin{equation}
TM\rtimes \Gamma\rightrightarrows M
\end{equation}
with structure maps
\begin{enumerate}
\item Source and target maps defined by
$$s(m,V,g)=m\cdot g^{-1},\,\,\, t(m,V)=m,$$
\item groupoid product
$$(m,V,g)\cdot (mg^{-1},W,h):=(m,V+d_{mg^{-1}}R_{g}(W),hg),$$
where $R_g:M\to M $ stands for the right action by $g$,
\item unit map
$$u(m)=(m,0_m,e),$$
where $e\in \Gamma$ is the neutral element for the group structure and $0_m\in T_mM$ is the origin vector, and
\item inverse map
$$(m,V)^{-1}:=(mg^{-1},-d_mR_{g^{-1}}(V),g^{-1}).$$
\end{enumerate}
For this groupoid we will usually consider the Schwartz groupoid convolution algebra of complex valued $C^\infty$- fiber rapidly decreasing functions with compact support in direction of the base and in the $\Gamma$-direction, $\mathscr{S}(TM\rtimes \Gamma)$. To be more precise an element in this algebra is a $C^\infty$, complex valued function $f$ on $TM\times \Gamma$ such that $f(V,g)=0$ for all but a finite number of $g\in \Gamma$ and such that the function
$f_g:TM\to \mathbb{C}$ defined by $f_g(V):=f(V,g)$ is in the Schwartz space $\mathscr{S}(TM)$ defined in \cite{Ca2} Definition 4.6. We can hence formally denote an element of  $\mathscr{S}(TM\rtimes \Gamma)$ as a finite sum
\begin{equation}
\sum_{g\in \Gamma}f_g\cdot g.
\end{equation}
The algebra product is given as the convolution product of $C_c^\infty(TM\rtimes \Gamma)$ where the groupoid under consideration is the groupoid $TM\rtimes \Gamma\rightrightarrows M$ as above, the product is well-defined since we are taking fiberwise Schwartz functions, see ref.cit for more details.
\begin{remark}
For a $\Gamma$-vector bundle $E\to M$ we can consider the analog algebra $\mathscr{S}(E\rtimes \Gamma)$.
\end{remark}
\end{itemize}

\begin{itemize}
\item {\bf The semi-direct groupoid $G_M^{tan} \rtimes \Gamma$:} First, let us remark that the diagonal action of $\Gamma$ on $M\times M$ sends the diagonal to the diagonal and hence the functorilaity of the deformation to the normal cone construction yields an action of $\Gamma$ on the tangent groupoid $G_M^{tan}$ of $M$ that unifies the action on $M\times M$ and the infinitesimal action on $TM$. The associated semi-direct product groupoid is denoted by
\begin{equation}
G_M^{tan}\rtimes \Gamma\rightrightarrows M\times [0,1]
\end{equation}
with groupoid structure given as a family of groupoids overs $[0,1]$ with the groupoid structure of $(M\times M)\rtimes \Gamma$ for $t\neq 0$ and with groupoid structure of $TM\rtimes \Gamma $ for $t=0$.

Again, this groupoid is a Lie groupoid and we could consider its convolution algebra but we will need a slightly larger algebra, indeed for this groupoid we consider the Schwartz type algebra $\mathscr{S}_c(G_M^{tan}\rtimes \Gamma)$ as an immediate generalisation of the Schwartz type algebra for the tangent groupoid defined in \cite{Ca2}. Formally an element in $\mathscr{S}_c(G_M^{tan}\rtimes \Gamma)$ is a $C^\infty$-function $f:G_M^{tan}\times \Gamma\to \mathbb{C}$ such that almost all but finite $f_g=f(\cdot,g):G^{tan}_M\to \mathbb{C}$ are zero and all of them belong to the vector space $\mathscr{S}_c(G_M^{tan})$ defined in \cite{Ca2} Section 4.2. Again, elements in $\mathscr{S}_c(G_M^{tan}\rtimes \Gamma)$ might be denoted as usual as sums
\begin{equation}
\sum_{g\in \Gamma}f_g\cdot g,
\end{equation}
with $f_g\in \mathscr{S}_c(G_M^{tan})$ (defined in Theorem 4.10 ref.cit.). The convolution product in this algebra is induced as classically by the groupoid product in $G_M^{tan}\rtimes \Gamma$.

The fundamental property of this algebra is to be an intermediate algebra
\begin{equation}
C_c^\infty(G_M^{tan}\rtimes \Gamma)\subset \mathscr{S}_c(G_M^{tan}\rtimes \Gamma) \subset C_r^*(G_M^{tan}\rtimes \Gamma)
\end{equation}
 such that the canonical evaluation morphisms yield a family of algebras
\begin{equation}
\mathscr{S}_c(G_M^{tan}\rtimes \Gamma)|_{t=0}=\mathscr{S}(TM\rtimes \Gamma) \,\, and \,\, \mathscr{S}_c(G_M^{tan}\times \Gamma)|_{t\neq 0}=C_c^\infty((M\times M)\rtimes \Gamma).
\end{equation}
In fact, there is a short exact sequence (further details in the next appendix)
\begin{equation}
\xymatrix{
0\ar[r]&J\ar[r]&\mathscr{S}_c(G_M^{tan}\times \Gamma)\ar[r]^-{e_0}&\mathscr{S}(TM\rtimes \Gamma)\ar[r]&0.
}
\end{equation}
As we will see in the next appendix the above kernel algebra $J$ has trivial periodic cyclic (co)homology groups and hence the evaluation at $t=0$ will induce isomorphisms in these (co)homology groups.
\end{itemize}

\begin{itemize}
\item {\bf The semi-direct deformation groupoid $D_f\rtimes \Gamma$:} The previous example is a particular case of the present example. To introduce it, let $f:M\longrightarrow N$ be a $\Gamma$-equivariant $C^\infty$-map between two $\Gamma$-manifolds $M$ and $N$, we consider the $\Gamma$-vector bundle
$T_f:=TM\oplus f^*TN$ over $M$ and the deformation groupoid
\begin{equation}
D_f:TM\oplus f^*TN\bigsqcup(M\times M \times N) \times (0,1]\rightrightarrows f^*TN\bigsqcup (M\times N) \times (0,1]
\end{equation}
given as the normal groupoid associated to the $\Gamma$-map
$$M\stackrel{\Delta\times f}{\longrightarrow}M\times M \times N.$$
Here the $\Gamma$-action on $M\times M \times N$ preserves the image of $M$ under the previous map $\Delta\times f$ and hence, by functoriality of the deformation to the normal cone construction, we have an action of $\Gamma$ on $D_f$ inducing a semi-direct product Lie groupoid
\begin{equation}
D_f\rtimes \Gamma \rightrightarrows f^*TN\bigsqcup (M\times N) \times (0,1]
\end{equation}
whose groupoid structure is given fiberwisely over $[0,1]$ as the groupoid structure of
$$((M\times M) \times N)\rtimes \Gamma\rightrightarrows M\times N$$
for $t\neq 0$, and as the groupoid structure
$$(TM\oplus f^*TN)\rtimes \Gamma \rightrightarrows f^*TN$$
at $t=0$.
We consider the Schwartz type algebra $\mathscr{S}_c(D_f\rtimes \Gamma)$ as defined in \cite{Ca2} for general deformation groupoids and whose fundamental property is to be an intermediate algebra
\begin{equation}
C_c^\infty(D_f\rtimes \Gamma)\subset \mathscr{S}_c(D_f\rtimes \Gamma) \subset C_r^*(D_f\rtimes \Gamma)
\end{equation}
such that the canonical evaluation morphisms yield a family of algebras
\begin{equation}
\mathscr{S}_c(D_f\rtimes \Gamma)|_{t=0}=\mathscr{S}((TM\oplus f^*TN)\rtimes \Gamma) \,\, and \,\, \mathscr{S}_c(D_f\rtimes \Gamma)|_{t\neq 0}=C_c^\infty(((M\times M)\times N)\rtimes \Gamma).
\end{equation}
In fact, there is a short exact sequence (further details in next appendix)
\begin{equation}
\xymatrix{
0\ar[r]&J\ar[r]&\mathscr{S}_c(D_f\rtimes \Gamma)\ar[r]^-{e_0}&\mathscr{S}((TM\oplus f^*TN)\rtimes \Gamma)\ar[r]&0.
}
\end{equation}
As we will see in the next appendix, the above kernel algebra $J$ has trivial Cyclic periodic (co)homology groups and hence the evaluation at $t=0$ will induce isomorphisms in these (co)homology groups.
\end{itemize}

\section{Some Schwartz type algebras for $\Gamma$-vector bundles}\label{appendixB}

We already introduced in the last appendix some algebras that use Schwartz functions in certain way. In this appendix we want to explain the reason of using Schwartz algebras in the present article and some important facts related to these algebras.

Given a $\Gamma$-proper equivariant vector bundle $E\to M$ there are two kinds of Schwartz algebras considered in this article, as sets they are the same, the vector space of complex valued $C^\infty$- fiber rapidly decreasing functions with compact support in direction of the base and in the $\Gamma$-direction, denoted by $\mathscr{S}(E\rtimes \Gamma)$. An element in this vector space is a complex valued function $f$ on $E\times \Gamma$ such that $f(V,g)=0$ for all but a finite number of $g\in \Gamma$ and such that the function
$f_g:E\to \mathbb{C}$ defined by $f_g(V):=f(V,g)$ is in the Schwartz space $\mathscr{S}(E)$ defined in \cite{Ca2} Definition 4.6. We can formally denote an element of  $\mathscr{S}(E\rtimes \Gamma)$ as a finite sum
\begin{equation}
\sum_{g\in \Gamma}f_g\cdot g.
\end{equation}
Now, there are two different algebra products in these spaces, depending on which groupoid one is considering, indeed if one considers the semi-direct product groupoid
$$E\rtimes \Gamma \rightrightarrows M$$
associated to the $\Gamma$-vector bundle structure then there is a convolution product $(\mathscr{S}(E\rtimes \Gamma),*)$ using this groupoid structure. On the other hand one might consider the action groupoid
$$E\rtimes \Gamma\rightrightarrows E$$
associated to the $\Gamma$-space $E$, in this case there is a punctual product algebra $(\mathscr{S}(E\rtimes \Gamma),\cdot)$ that uses indeed the unit  groupoid structure. These two algebras are not the same at all. In the context, we use both examples, for the semi-direct product groupoid
$$TM\rtimes \Gamma\rightrightarrows M,$$
and for the action groupoid
$$T^*M\rtimes \Gamma\rightrightarrows T^*M,$$
or for the bundles $T_f$ and $T_f^*$ respectively. As in the classical case when $\Gamma$ is the trivial group we have a Fourier type algebra isomorphism
$$(\mathscr{S}(TM\rtimes \Gamma),*)\cong (\mathscr{S}(T^*M\rtimes \Gamma),\cdot),$$
which indeed is a particular case of the following proposition.
\begin{proposition}
Given a $\Gamma$-proper equivariant vector bundle $E\to M$, we have the following two properties:
\begin{enumerate}
\item The fiberwise Fourier transform gives an algebra isomorphism
\begin{equation}
(\mathscr{S}(E\rtimes \Gamma),*)\cong (\mathscr{S}(E^*\rtimes \Gamma),\cdot).
\end{equation}
\item The convolution algebra $(\mathscr{S}(E\rtimes \Gamma),*)$ is stable under holomorphic calculus as a subalgebra of $(C_r^*(E\rtimes \Gamma),*)$.
\end{enumerate}
\end{proposition}

The proof of the above two statements is quite classical, for instance one can follow the same lines as the proof of Proposition 4.5 in \cite{Ca4}.

Now, one important consequence for us is that, in the above situation, we have the following isomorphisms in $K$-theory

\begin{equation}
K^*_\Gamma(E^*)\cong K_*(\mathscr{S}(E^*\rtimes \Gamma),\cdot)\cong K_*(\mathscr{S}(E\rtimes \Gamma),*)\cong K_*(C^*_r(E\rtimes \Gamma),*) ,
\end{equation}
where the first group above stands for the $\Gamma$-equivariant topological $K-$theory for the total space $E^*$.

When applied to the bundle $E=TM$ (or $T_f$), as in the text, the series of isomorphisms above is very important for us, since for a $\Gamma$-proper manifold $M$ the kind of elliptic operators we want to consider have principal symbol classes in the groups on the left but the deformation groupoid techniques give total symbol classes in groups as in the right hand side. On the other hand, using only $C^*$-algebras would be very limited for us, since we apply as well isomorphisms as above but in periodic cyclic (co)homology, in particular the ones induced from the Fourier algebra isomorphism
\begin{equation}
HP_*(\mathscr{S}(E^*\rtimes \Gamma),\cdot)\cong HP_*(\mathscr{S}(E\rtimes \Gamma),*)
\end{equation}
and
\begin{equation}
HP^*(\mathscr{S}(E^*\rtimes \Gamma),\cdot)\cong HP^*(\mathscr{S}(E\rtimes \Gamma),*).
\end{equation}
\section{Deformation morphisms in periodic cyclic (co)homology}\label{appendixdefindices}

Let $f:M\longrightarrow N$ be a $\Gamma$-equivariant $C^\infty$-map between two $\Gamma$-manifolds $M$ and $N$. We consider the $\Gamma$-vector bundle
$T_f:=TM\oplus f^*TN$ over $M$. We assume $M$ is $\Gamma$-proper and that $T_f$ admits a $\Gamma$-proper, $\Gamma$-spin$^c$ structure. We are going to use deformation groupoids to construct a morphism
\begin{equation}
I_f:HP_*(\mathscr{S}(T_f\rtimes \Gamma))\longrightarrow HP_*(C_c^{\infty}((M\times M \times N)\rtimes \Gamma)).
\end{equation}
Consider the deformation groupoid
\begin{equation}
D_f:=TM\oplus f^*TN\bigsqcup(M\times M \times N) \times (0,1]
\end{equation}
over
\begin{equation}
f^*TN\bigsqcup (M\times N) \times (0,1]
\end{equation}
given as the normal groupoid associated to the $\Gamma$-map
$$M\stackrel{\Delta\times f}{\longrightarrow}M\times M \times N.$$
As we already discussed above, the action of $\Gamma$ on $(M\times M)\times N$ (diagonal on $M\times M$) leaves invariant the image of $M$ by $\Delta \times f$ and then we have, by functoriality of the deformation to the normal cone construction, an induced semi-direct product groupoid
$$D_f\rtimes \Gamma\rightrightarrows f^*TN\bigsqcup (M\times N) \times (0,1].$$

Let us describe two invariant open subset covering $D_f$ and hence giving a decomposition of $D_f\rtimes \Gamma$. Under the hypothesis above on $T_f$ (that it is identified with the normal bundle associated to $\Delta \times f$) we can assume that we have a global $\Gamma$- tubular neighborhood isomorphism
$$
\xymatrix{
T_f\ar[r]^-{\theta}& \mathcal{V}\subset M\times M \times N,
}
$$
meaning that $\theta$ is a diffeomorphism from $T_f$ to an open neighborhood $\mathcal{V}$ of $M$ in $M\times M \times N$ sending the zero section to $M$ by the identity. The existence of such a global invariant tubular neighborhood can be obtained for example from a $\Gamma$-invariant metric on $T_f$ and this is possible to obtain since the action is proper.

We have the following $\Gamma$-invariant chart for $D_f$
\begin{equation}\label{psichart}
T_f\times [0,1]\stackrel{\psi}{\longrightarrow} \widetilde{\mathcal{V}}\subset D_f
\end{equation}
explicitly given by
$$\psi(V,0)=(V,0),$$
and
$$\psi(V,\epsilon)=(\theta(\epsilon\cdot V),\epsilon)$$
for $\epsilon\neq 0$. The semidirect product groupoid $D_f\rtimes \Gamma$ might be covered with two open subsets as
$$\widetilde{\mathcal{V}}\rtimes \Gamma \bigcup ((M\times M) \times N \times (0,1])\rtimes \Gamma.$$
By construction of the space $\mathscr{S}_c(D_f)$, we have that
$$\mathscr{S}_c(D_f\rtimes \Gamma)=\mathscr{S}(\widetilde{\mathcal{V}}\rtimes \Gamma)+C_c^\infty(((M\times M) \times N \times (0,1])\rtimes \Gamma).$$
Notice that the intersection $\widetilde{\mathcal{V}}\rtimes \Gamma \bigcap ((M\times M) \times N \times (0,1])\rtimes \Gamma$ identifies with $\mathcal{V}\times (0,1]$

Since periodic cyclic (co)homology has a six term Mayer-Vietoris exact sequence and since the periodic cyclic (co)homology groups of
$C_c^\infty(((M\times M) \times N \times (0,1])\rtimes \Gamma)$ and of $\mathscr{S}(\mathcal{V} \times(0,1])$ vanish by homotopy invariance of periodic cyclic (co)homology, we obtain that the morphisms
\begin{equation}
HP^*(\mathscr{S}_c(D_f\rtimes \Gamma))\stackrel{i_{\mathcal{V}}^*}{\longrightarrow} HP^*(\mathscr{S}(\widetilde{\mathcal{V}}\rtimes \Gamma))
\end{equation}
and
\begin{equation}
HP_*(\mathscr{S}(\widetilde{\mathcal{V}}\rtimes \Gamma))\stackrel{(i_{\mathcal{V}})_*}{\longrightarrow} HP_*(\mathscr{S}_c(D_f\rtimes \Gamma)),
\end{equation}
induced from the canonical inclusion algebra morphism
\begin{equation}
i_{\mathcal{V}}:\mathscr{S}(\widetilde{\mathcal{V}}\rtimes \Gamma)\longrightarrow \mathscr{S}_c(D_f\rtimes \Gamma),
\end{equation}
are isomorphisms. Note that the chart $\psi$ in (\ref{psichart}) induces an isomorphism
\begin{equation}
HP^*(\mathscr{S}(\widetilde{\mathcal{V}}\rtimes \Gamma)) \stackrel{\widetilde{\psi}}{\longrightarrow}HP^*(\mathscr{S}((T_f\times [0,1])\rtimes \Gamma)),
\end{equation}
and by homotopy invariance, the morphism
\begin{equation}
\alpha^*:HP^*(\mathscr{S}((T_f\times [0,1])\rtimes \Gamma))\longrightarrow HP^*(\mathscr{S}(T_f\rtimes \Gamma)),
\end{equation}
induced from the canonical morphism (constant in $\epsilon\in [0,1]$)
$$\alpha:\mathscr{S}(T_f\rtimes \Gamma)\to \mathscr{S}((T_f\times [0,1])\rtimes \Gamma),$$
is also an isomorphism. A direct and simple computation shows that the isomorphism
\begin{equation}
\xymatrix{
HP^*(\mathscr{S}_c(D_f\rtimes \Gamma))\ar[r]^-{\alpha^*\circ \widetilde{\psi}\circ i_{\mathcal{V}}^*}& HP^*(\mathscr{S}(T_f\rtimes \Gamma))
}
\end{equation}
has as left and right inverse the morphism (and hence isomorphism)
\begin{equation}
HP^*(\mathscr{S}(T_f\rtimes \Gamma))\stackrel{e_0^*}{\longrightarrow} HP^*(\mathscr{S}_c(D_f\rtimes \Gamma)).
\end{equation}

Therefore we can define the {\it deformation morphism}
\begin{equation}
I_f:HP_*(\mathscr{S}(T_f\rtimes \Gamma))\longrightarrow HP_*(C_c^{\infty}((M\times M \times N)\rtimes \Gamma))
\end{equation}
as the composition of the inverse of the morphism induced by the evaluation at zero
$$
\xymatrix{
HP_*(\mathscr{S}_c(D_f\rtimes \Gamma))\ar[r]^-{(e_0)_*}_-\cong & HP_*(\mathscr{S}(T_f\rtimes \Gamma)),
}
$$
followed by the morphism induced by the evaluation at $t=1$
$$HP_*(\mathscr{S}_c(D_f\rtimes \Gamma))\stackrel{(e_1)_*}{\longrightarrow} HP_*(C_c^{\infty}((M\times M \times N)\rtimes \Gamma)).$$

\begin{remark}
Notice that we have not only proved that $e_0^*$ above is an isomorphism but we have exhibit an explicit expression for the inverse by ``localizing" at an invariant neighborhood. This inverse will be explicitly used in the computations, it also gives a conceptual explanation to many analytically founded index formulas in the literature. Similar arguments work for periodic cyclic homology.
\end{remark}

\section{The Tu-Xu isomorphism}\label{appendixTuXu}

Given a discrete group $\Gamma$ acting properly by diffeomorphisms on a smooth manifold $M$, Tu and Xu constructed\footnote{In fact for more general proper etale groupoids and with the possible extra data of a twisting but we do not need these generalities here.}  (for $*=0,1$ mod $2$ in \cite{TuXuChern}
Section 4) an isomorphism
\begin{equation}\label{tau1appendix}
\xymatrix{
HP_*(C_c^\infty(M\rtimes \Gamma))\ar[r]^-{\tau}_-\cong& H^*_{c,per}(\Lambda(M\rtimes \Gamma))
}
\end{equation}
where $H^*_{c,per}(\Lambda(M\rtimes \Gamma))$ is a periodic version of the compactly supported de Rham cohomology groups of the so-called inertia groupoid that we now briefly recall: consider the manifold
\begin{equation}
S(M\rtimes \Gamma):=\{(x,\gamma)\in M\times \Gamma: x\cdot \gamma=x\}
\end{equation}
and the canonical action of $\Gamma$ on $S(M\rtimes \Gamma)$ by conjugation. The associated action Lie groupoid is denoted by
\begin{equation}
\Lambda(M\rtimes \Gamma)\rightrightarrows S(M\rtimes \Gamma)
\end{equation}
and is called the inertia groupoid associated to the action.

In this appendix we briefly recall the Tu-Xu isomorphism and how a small modification of it gives an isomorphism
\begin{equation}
\xymatrix{
HP_*(C_c^\infty(M\rtimes \Gamma))\ar[r]^-{TX}_-\cong& H^*_{\Gamma,deloc}(M)
}
\end{equation}
fitting in a commutative diagram of isomorphisms
\begin{equation}
\xymatrix{
HP_*(C_c^\infty(M\rtimes \Gamma))\ar[rrd]_-{TX}^-\cong \ar[rr]^-{\tau}_-\cong&& H^*_{c,per}(\Lambda(M\rtimes \Gamma))\ar[d]^-\cong\\
&&H^*_{\Gamma,deloc}(M)
}
\end{equation}
where the vertical isomorphism above is induced, as we will see below, from canonical closed Lie groupoids inclusions
$$M_g\rtimes \Gamma_g\hookrightarrow \Lambda(M\rtimes \Gamma).$$
First of all let us recall what the groups $H^*_{c,per}(\Lambda(M\rtimes \Gamma))$ stand for. We remark that we are not using the same notation as Tu and Xu. As we mentioned above, Tu and Xu isomorphism is more general, it is done in the context of twisted orbifolds, in our case the twisting is trivial and the orbifolds are particularly given as quotients $M/\Gamma$. In our setting, following Tu and Xu, the groups $H^*_{c,per}(\Lambda(M\rtimes \Gamma))$ are defined as the cohomology of the complex
\begin{equation}
\left(\Omega_c^*(\Lambda(M\rtimes \Gamma)((u)),\nabla \right),
\end{equation}
where $u$ is a formal variable of degree $-2$, $((u))$ stands for formal Laurent series in $u$ and $\nabla:\Omega_c^*(S(M\rtimes \Gamma))\to \Omega_c^{*+1}(S(M\rtimes \Gamma))$ is a $\Gamma$-invariant covariant differential associated to a canonical flat connection of the trivial line bundle $S(M\rtimes \Gamma)\times \mathbb{C}\to S(M\rtimes \Gamma)$. In fact in this case, since the action is proper, $\Omega_c^*(\Lambda(M\rtimes \Gamma)=\Omega_c^*(S(M\rtimes \Gamma))^\Gamma$ where the latter is the space of all $\Gamma$-invariant forms, see Corollary 3.2 in \cite{TuXuChern} for further details. In particular, in our non-twisted case, we have
\begin{equation}
H^*_{c,per}(\Lambda(M\rtimes \Gamma))\cong \bigoplus_{k=*,\,mod\, 2} H^k_c(\Lambda(M\rtimes \Gamma)),
\end{equation}
however we will use the terminology with the formal variable $u$ to describe such groups since it makes easier to follow Tu and Xu construction. The morphism $\tau$ of (\ref{tau1appendix}) above depends  \`a priori on the choice of a covariant differential $\nabla:\Omega_c^*(M\rtimes \Gamma)\to \Omega_c^{*+1}(M\rtimes \Gamma)$ and a $2$-curvature form $B\in \Omega(M)$ (considered as an element of $\Omega^*(M\rtimes \Gamma)$ since $M$ is an open and close submanifold of $M\rtimes \Gamma$), Tu and Xu explored in their paper the relations when these choices are changed, we do not need this in our paper since we will be always taking the canonical connection associated to the trivial $S^1$-central extension over $M\rtimes \Gamma$. Denoting by $A=C_c^\infty(M\rtimes \Gamma)$, the linear map
$$\tau:CC_k(A)\to \Omega_c^*(\Lambda(M\rtimes \Gamma)((u))$$
is defined by
\begin{equation}\label{taudefinition}
\tau(\tilde{a}_0,a_1,...,a_k)=\int_{\Delta_k}Tr(\tilde{a}_0*e^{2\pi ius_0B}*\nabla(a_1)*\cdots*\nabla(a_k)*e^{2\pi ius_kB})ds_0ds_1\cdots ds_k,
\end{equation}
where the integration is over the $k$-simplex $\Delta_k=\{(s_0,...,s_k):s_i\geq 0, \forall i, \sum s_i=1\}$ and
where $\nabla:\Omega_c^*(M\rtimes \Gamma)\to \Omega_c^{*+1}(M\rtimes \Gamma)$ and $B\in \Omega(M)$ a $2$-curvature form (considered as an element of $\Omega^*(M\rtimes \Gamma)$ since $M$ is an open and close submanifold of $M\rtimes \Gamma$) as mentioned above, the product $*$ denotes a canonical convolution product on $\Omega_c^*(M\rtimes \Gamma)$ that makes it an associative graded algebra (Lemma 4.1 in \cite{TuXuChern}) and, most importantly,
$$Tr: \Omega_c^*(M\rtimes \Gamma)\to \Omega_c^*(\Lambda(M\rtimes \Gamma))$$
is a trace map constructed in Section 4.2 ref.cit that we briefly recall. It is called a trace map since it satisfies the classic ``trace" like properties (Proposition 4.5 ref.cit.)
that imply, as classically, that the above linear map (its extension by linearity to Laurent series in $u$)
$$\tau:CC_k(A)((u))\to \Omega_c^*(\Lambda(M\rtimes \Gamma))((u))$$
induces a homomorphism in cohomology (Proposition 4.6 ref.cit.)
$$\tau:HP_*(A)\to H^*_{c,per}(\Lambda(M\rtimes \Gamma)).$$
Let $\omega\in \Omega_c^*(M\rtimes \Gamma)$, consider $i:S(M\rtimes \Gamma)\to M\rtimes \Gamma$ the closed inclusion
and the form $i^*\omega\in \Omega_c^*(S(M\rtimes \Gamma))$, Tu and Xu defined
$$Tr(\omega):=\sum_{\gamma\in\Gamma}t^*_\gamma(i^*\omega),$$
where $t_\gamma:S(M\rtimes \Gamma)\to S(M\rtimes \Gamma)$ is the diffeomorphism given by the action by conjugation of $\gamma$ on $S(M\rtimes \Gamma)$.
A direct computation shows that $Tr(\omega)$ is a $\Gamma$-invariant form on $S(M\rtimes \Gamma)$ with compact support.
Now, let us fix $g\in \Gamma$ of finite order and consider the closed embedding
$$j_g:M_g\to S(M\rtimes \Gamma),$$
that sends $m \mapsto (m,g)$. It extends to a closed groupoid embedding
$$M_g\rtimes \Gamma_g\to S(M\rtimes \Gamma)\rtimes \Gamma,$$
where on the right the action is by conjugation. In particular if we let $Tr_g(\omega):=j_g^*Tr(\omega)$, it defines a $\Gamma_g$-invariant form on $M_g$ with compact support. It is also simple to check that $Tr_g$ satisfies the ``trace" like properties of Proposition 4.5 in ref.cit. Hence we have a homomorphism in cohomology
$$\tau_g:HP_*(A)\to \bigoplus_{k=*,\,mod\, 2}H^k_c(M_g\rtimes \Gamma_g)$$
induced from a chain morphism
$$\tau_g: CC_k(A)((u))\to \Omega_c^*(M_g\rtimes \Gamma_g)((u))$$
defined as (\ref{taudefinition}) above but using instead the trace map $Tr_g$.
We can hence consider the induced morphism
\begin{equation}
TX:=\bigoplus_{g\in \langle\Gamma\rangle^{fin}}\tau_g: HP_*(A)\to H^*_{\Gamma,deloc}(M)
\end{equation}
that fits by construction in the following commutative diagram
\begin{equation}
\xymatrix{
HP_*(C_c^\infty(M\rtimes \Gamma))\ar[rrd]_-{TX} \ar[rr]^-{\tau}&& H^*_{c,per}(\Lambda(M\rtimes \Gamma))\ar[d]\\
&&H^*_{\Gamma,deloc}(M),
}
\end{equation}
where the vertical map is induced from the embeddings $j_g$ as above.
In Section 4.4 of their paper, Tu and Xu proved that the morphism $\tau$ above is an isomorphism. The idea  is that $HP_*(C_c^\infty(M\rtimes \Gamma))$ and $H^*_{c,per}(\Lambda(M\rtimes \Gamma))$ agree locally, since an orbifold is locally a crossed product of a manifold by a finite group, and each of these cohomology funtors admits Mayer-Vietoris sequences and six term exact sequences, hence they agree globally. By following step by step the proof of Tu and Xu but using instead $TX$ and $H^*_{\Gamma,deloc}(M)$ onecan prove that $TX$ is an isomorphism as well. This is of course in complete accordance, it gives another proof for, to all known computations of the periodic cohomology groups for these kind of groupoids, see for instance \cite{BN,Cra,Ponge}. Notice that in particular we obtain that the vertical map in the last commutative diagram is as well an isomorphism.

\begin{remark}
Let $E\to M$ be a $\Gamma$-equivariant, proper vector bundle. By considering fiberwise rapidly decreasing forms $\Omega^*_{sch}(E\rtimes \Gamma)$ and $\Omega^*_{sch}(\Lambda (E\rtimes \Gamma))$ we can follow step by step of Tu and Xu construction above to get an isomorphism
\begin{equation}
TX: HP_*(\mathscr{S}(E\rtimes \Gamma))\stackrel{\cong}{\longrightarrow} H^*_{\Gamma,deloc}(E).
\end{equation}
\end{remark}

\section{Thom isomorphism in periodic cyclic (co)homology and in delocalised cohomology}\label{appendixThom}

Let $M$ be a $\Gamma$-manifold and let $E\rightarrow M$ be a $\Gamma$-oriented vector bundle over $M$. In the present paper all our vector bundles admit by hypothesis a $\Gamma$-spin$^c$ structure, we will work in this setting.


We want to describe the Thom isomorphisms
\begin{equation}\label{cohomThomisodeloc}
H^*_{\Gamma,deloc}(M)\stackrel{Th}{\longrightarrow}H^*_{\Gamma,deloc}(E)
\end{equation}
and
\begin{equation}\label{homThomisodeloc}
H_*^{\Gamma,deloc}(E)\stackrel{Th}{\longrightarrow}H_*^{\Gamma,deloc}(M).
\end{equation}
We will follow the book \cite{BGNX} Chapter 4 where more general cases of Thom isomorphisms are treated.

Following Behrend, Ginot, Noohi and Xu in \cite{BGNX} (Chapter 4), there are forms $U_{E_g}\in \Omega^*(E_g\rtimes \Gamma_g)$
such that the maps
$$\alpha\mapsto \pi_{E_g}^*\alpha\cup U_{E_g}$$
induce an isomorphism
$$H^*_{\Gamma,deloc}(M)\stackrel{Th}{\longrightarrow}H^*_{\Gamma,deloc}(E).$$
The duality discussed in Section \ref{PDuality} induces an isomorphism
$$H_*^{\Gamma,deloc}(E)\stackrel{Th}{\longrightarrow}H_*^{\Gamma,deloc}(M),$$
given in terms of currents as
$$C\mapsto Th(C)$$
where componentwisely.
$$Th(C)(\alpha):=C(\pi_{E_g}^*(\alpha)\cup U_{E_g}).$$

Using the Baum-Connes isomorphism (Theorem 7.14 in \cite{BCfete} or the dual of the Tu-Xu isomorphism above)
$$H_*^{\Gamma,deloc}(M)\stackrel{BC}{\longrightarrow} HP^*(C_c^\infty(M\rtimes \Gamma))$$
and the analog for $E$ (as a unit groupoid), the Thom isomorphism (\ref{homThomisodeloc}) induces an unique isomorphism in periodic cyclic cohomology
\begin{equation}
HP^*(\mathscr{S}(E\rtimes \Gamma))\stackrel{Th}{\longrightarrow} HP^*(C_c^\infty(M\rtimes \Gamma)).
\end{equation}
Now, we could do the same in periodic cyclic homology but we won't, the reason is that we want the Thom isomorphism to commute with Chern-Connes character and to satisfy properties as its K-theoretical version, for this reason we let
\begin{equation}
HP_*(C_c^\infty(M\rtimes \Gamma))\stackrel{Th}{\longrightarrow} HP_*(\mathscr{S}(E\rtimes \Gamma))
\end{equation}
as the unique isomorphism commuting (up to tensoring with $\mathbb{C}$) with the Chern-Connes character, we call it as well the Thom isomorphism even if we know that it is more a twisted Thom isomorphism.

\begin{remark}
In this section we have defined some Thom isomorphisms without entering too much in the details of their explicit definitions, the reason is that in the present article we only need some formal properties these morphisms.
\end{remark}

\bibliographystyle{plain}
\bibliography{bibliographie}

\end{document}